\documentclass[a4paper,11pt]{amsart}
\usepackage{amsmath}
\usepackage{url,graphicx}
\usepackage{amssymb, color, pstricks-add, manfnt}
\usepackage{amsmath, amsthm,  amsfonts}
\usepackage{mathrsfs}
\usepackage{cite}
\usepackage{enumerate}
\usepackage[none]{hyphenat}
\usepackage{multirow}
\usepackage{xcolor}
\usepackage{colortbl}
\usepackage{tabularx}

\theoremstyle{plain}
\newtheorem{Theorem}{Theorem}[section]
\newtheorem{Lemma}{Lemma}[section]
\newtheorem{Proposition}{Proposition}[section]

\newtheorem{Corollary}{Corollary}[section]
\newtheorem{Definition}{Definition}[section]

\theoremstyle{remark}
\newtheorem{remark}{Remark}

\numberwithin{equation}{section}
\numberwithin{figure}{section}
\numberwithin{remark}{section}
\usepackage{autobreak}%让长公式自动换行
\allowdisplaybreaks%跨页自动断页

\usepackage{hyperref}
\usepackage{cite}

\begin{document}
\begin{sloppypar}

	\title{Regularity of solutions for degenerate or singular fully nonlinear integro-differential equations}
	
	\author{Jiangwen Wang}
	\address{School of Mathematics and Shing-Tung Yau Center of Southeast University, Southeast University, Nanjing 211189, P.R. China}
	\email{jiangwen\_wang@seu.edu.cn}	

	\author{Feida Jiang$^*$}
	\address{School of Mathematics and Shing-Tung Yau Center of Southeast University, Southeast University, Nanjing 211189, P.R. China; Shanghai Institute for Mathematics and Interdisciplinary Sciences, Shanghai 200433, P.R. China}
	\email{jiangfeida@seu.edu.cn}	
	
	\subjclass[2010]{35D40, 35R11, 35J70, 35J75}

	%\communicated{}
	\date{\today}
	\thanks{*corresponding author}
	
		\keywords{Integro-differential operator; degenerate/singular equation; viscosity solutions; higher regularity}

	\begin{abstract}
	We study a series of regularity results for solutions to a                                                                      degenerate or singular fully nonlinear integro-differential equation of the form
			$$- \big( \sigma_{1}(|Du|)  + a(x) \sigma_{2}(|Du|)    \big) \mathcal{I}_{\tau}(u,x) = f(x).$$
			In the degenerate case, we establish borderline regularity,  provided the inverse of the degeneracy law $ \sigma_{2}$ is Dini-continuous. In addition, we show Schauder-type higher regularity at local extremum points for a specific non-local degenerate equation. In the singular case, we establish H\"{o}lder continuity of the gradient for solutions to a general non-local equation.

It is noteworthy that these results are new even in the case $ a(x) \equiv 0 $. Finally, as a byproduct of the borderline regularity analysis, we demonstrate how our methods can be applied to study of the corresponding regularity for a class of degenerate non-local normalized $ p$-Laplacian equations.
\end{abstract}

	\maketitle		
\tableofcontents

\section{Introduction}
Integro-differential operators are often referred to as non-local operators in the literature. This class of problems arises in various mathematical modeling contexts, such as image processing and financial payoff models; see, for instance \cite{GVR16, CE16, EGE12} and the references therein. For example, linear non-local operators naturally emerge in the study of discontinuous stochastic processes, while nonlinear integro-differential equations are prevalent in the analysis of stochastic differential games and optimal stopping problems, see \cite{XX24}.

Integro-differential operators have been extensively studied over the years by many authors. In \cite{LL09}, Caffarelli and Silvestre established a series profound results extending the theory of second-order operators, including the  Aleksandrov-Bakelman-Pucci (ABP) estimate, $ C^{\alpha} $ regularity, $ C^{1,\alpha} $ estimates, and more. Subsequently, in \cite{LL11}, they extended their previous results to non-translation-invariant equations using perturbative techniques. Around the same time, Barles {\it et.al} \cite{BCI11} obtained H\"{o}lder estimates via the Ishii-Lions viscosity method for a broad class of elliptic and parabolic integro-differential equations that incorporate both second-and first-order terms. For results on Lipschitz regularity within a general framework, we refer the interested reader to \cite{BCCI12}.

  \vspace{2mm}

In particular, in this paper we consider a series of regularity results for solutions to the double-phase degenerate or singular fully nonlinear integro-differential equation of the form
\begin{equation}\label{Intro:eq1}
- \bigg( \sigma_{1}(|Du|)  + a(x) \sigma_{2}(|Du|)    \bigg) \mathcal{I}_{\tau}(u,x) = f(x) \ \ \text{in} \ \ B_{1},
\end{equation}
where $ \sigma_{i} : \mathbb{R}_{+} \rightarrow \mathbb{R}_{+}, i=1,2 $, are degeneracy rates, $B_1$ is the unit ball in $\mathbb{R}^n$, and $ \mathcal{I}_{\tau}  $ is a fully nonlinear elliptic integro-differential operator
\begin{equation}\label{I_tau}
 \mathcal{I}_{\tau}(u,x)  = \inf_{i} \sup_{j} \left[P.V. \int_{\mathbb{R}^{n}} \left( u(y) - u(x)   \right) K_{ij}(x-y) dy\right], 
\end{equation}
with a two-parameter family of symmetric kernels $ \{K_{ij} \}\subseteq \mathcal{K}_{0}$ for $i\in \mathcal{A}$ and $j\in \mathcal{B}$, and $P.V.$ stands for the Cauchy principal value.
Here, $\mathcal{A}, \mathcal{B}$ are arbitrary index sets, and $\mathcal{K}_0$ denotes the family of symmetric kernels $ K : \mathbb{R}^{n} \setminus \{0\} \rightarrow \mathbb{R}_{+}  $ satisfying
\begin{equation}\label{bd for K}
 \lambda \frac{C_{n,\tau}}{|x|^{n+\tau}} \leq K(x) \leq \Lambda \frac{C_{n,\tau}}{|x|^{n+\tau}}, 
\end{equation}
for the order $\tau$ and certain constants $ 0< \lambda \leq \Lambda < \infty $ and $ C_{n,\tau} > 0 $ is a constant. Note that \eqref{I_tau} is well defined for function $ u \in C^{1,1}(B_{\delta}(x)) \cap L^{1}_{\tau}(\mathbb{R}^{n}) $ for some $ \delta >0 $, where
\begin{equation*}
||u||_{L^{1}_{\tau}(\mathbb{R}^{n})} := \int_{\mathbb{R}^{n}} |u(x)| \frac{1}{1+|x|^{n+\tau}} dx <+\infty .
\end{equation*}

We first present a list of appropriate assumptions on $\tau, u, f, a$ and $\sigma_i$, $ i=1, 2 $.

\vspace{2mm}

\label{A1}  {\bf (A1) (order of the operator).} $ \tau \in  (1, 2)$.

\vspace{2mm}

\label{A2} {\bf (A2) (growth condition at infinity).} $u \in L_{\tau}^{1}(\mathbb{R}^{n})$.

\vspace{2mm}

\label{A3} {\bf (A3) (regularity of source term).} $ f \in L^{\infty}(B_{1}) \cap C^{0}(B_{1})   $.

\vspace{2mm}

\label{A4} {\bf (A4) (nonnegativity and regularity of $ a(x)$).} $ 0 \leq a(x) \in C^{0}(B_{1}) $.

\vspace{2mm}

\label{A5} {\bf (A5) (convergence of the operator).} Let $ \{K_{ij}\} $ be a collection of kernels of in $ \mathcal{K}_{0} $. There exists a modulus of continuity $ \omega $ and $ \{k_{ij}\} \in [\lambda, \Lambda] $ satisfying
\begin{equation*}
|K_{ij}(x)|x|^{n+\tau} - k_{ij} | \leq \omega(|x|) \ \text{for all} \ i,j \ \text{and} \ |x|\leq1 . 
\end{equation*}

\vspace{2mm}

\label{A6} {\bf (A6) (monotonicity of $ \sigma_{i}, i=1,2 $).}

-\label{A6a} {\bf (A6a) (degenerate case).} $ \sigma_{i}: \mathbb{R}_{+} \rightarrow \mathbb{R}_{+}, i=1,2$, are continuous and monotone increasing with
\begin{equation*}
 \lim_{t \rightarrow 0} \sigma_{i}(t)=0, \ \ i=1,2.
\end{equation*}

-\label{A6b} {\bf (A6b) (singular case).} $ \sigma_{i}: \mathbb{R}_{+} \rightarrow \mathbb{R}_{+}, i=1,2$, are continuous and monotone decreasing with
\begin{equation*}\label{intro:eq4}
\lim_{t \rightarrow 0} \sigma_{i}(t) = c_{0} > 0, \ \ i=1,2.
\end{equation*}

Furthermore, in both cases, we assume
\begin{equation*}
 \sigma_{1}(t) \geq \sigma_{2}(t), \ \ t \in [0,1],
\end{equation*}
and, in particular,
\begin{equation*}
  \sigma_{1}(1)  \geq \sigma_{2}(1) \geq 1 .
\end{equation*}

\vspace{2mm}

\label{A7} {\bf (A7) (Dini continuity of the inverse of $\sigma_2$).} The function $ \sigma_{2} $ admits an inverse $ \sigma_{2}^{-1} $ that is Dini continuous, i.e.,
\begin{equation*}
  \int_{0}^{1} \frac{\sigma_{2} ^{-1}(t)}{t} dt < \infty.
\end{equation*}

\label{A8} {\bf (A8) (Non-collapsing of $ \sigma_{1} + a(\cdot)\sigma_{2} $).} $ \sigma_{1} + a(\cdot)\sigma_{2} \in \Sigma $, where $ \Sigma $ is a set of non-collapsing moduli of continuity.

Note that as in \cite[Denfintion 3]{PDEE22}, a set $ \Sigma $ of moduli of continuity is said to be non-collapsing if for all sequences $ (f_{k})_{k\in \mathbb{N}} \subset \Sigma  $, and all sequences of $ (a_{k})_{k\in \mathbb{N}} \subset (0,\infty)  $, we have
\begin{equation*}
 \lim_{k\rightarrow +\infty} f_{k}(a_{k}) = 0 \ \  \Rightarrow  \ \
 \lim_{k\rightarrow +\infty} a_{k} = 0.
\end{equation*}

\vspace{2mm}

Over the past decades, $ C^{1,\alpha} $ regularity of solutions to degenerate or singular elliptic equations has been the subject of intensive study in regularity theory, with one of the most widely studied prototypes given by
\begin{equation}\label{Intro:eq1p}
  |Du|^{\gamma} F(D^{2}u) = f(x) \ \ \text{in} \ \ B_{1},
\end{equation}
where $ \gamma > -1 $ and $ F $ is a uniformly elliptic operator; see \cite{ART15, CL13, AEGE21, YRZ21, C21, SSKS24}. It is noteworthy that the regularity theory for viscosity solutions to (\ref{Intro:eq1p}) plays a crucial role in the study of various free boundary problems, such as those of obstacle type \cite{1DV21, 2DV21}, the one-phase Bernoulli type \cite{DRRV23}, singular perturbation problems \cite{ART17}, and the free transmission problem \cite{HPRS24, C22, WYJ25}.

In a recent paper \cite{PDEE22}, the authors considered a degenerate fully nonlinear equation of the form
\begin{equation}\label{Intro:eq2}
  \sigma(|Du|) F(D^{2}u) = f(x) \ \ \text{in} \ \ B_{1},
\end{equation}
and established local $ C^{1} $ (borderline) regularity for viscosity solutions of (\ref{Intro:eq2}) under the assumption that $ \sigma : \mathbb{R}_{+} \rightarrow \mathbb{R}_{+} $ is a function whose inverse has a Dini-continuous modulus of continuity near the origin.

In the non-local setting, the first $C^{1, \alpha}$ regularity result was obtained by Prazeres and Topp in \cite{DE21}. They considered anon-local degenerate equation of the form
\begin{equation}\label{Intro:eq3}
-|Du|^{\gamma} \mathcal{I}_{\tau}(u,x) = f(x) \ \ \text{in} \ \ B_{1},
\end{equation}
where $ \gamma > 0 $ and the order $ \tau $ is sufficiently close to $ 2 $. They proved that viscosity solutions to (\ref{Intro:eq3}) belong to the class $ C^{1,\alpha} $. Their assumptions on $ \tau $ allowed them to utilize the following canceling property:
\begin{equation*}
|Du|^{\gamma} F(D^{2}u) = 0  \ \ \Rightarrow \ \ F(D^{2}u) = 0.
\end{equation*}
As observed in \cite[Section 3]{DE21}, this property plays a crucial role in obtaining interior regularity estimates for solutions of (\ref{Intro:eq3}).

It is also worth mentioning that in the recent work \cite{PDM24}, the authors examined $ C^{1,\alpha} $ regularity of solutions to fully nonlinear non-local equations with double degeneracy or singularity of the form
\begin{equation*}
-\big( |Du|^{\gamma_{1}} +a(x)|Du|^{\gamma_{2}} \big)\mathcal{I}_{\tau}(u,x) = f(x) \ \ \text{in} \ \ B_{1},
\end{equation*}
where $ -1 < \gamma_{1} \leq \gamma_{2} < \infty  $ and $ a(x) \in C^{0}(B_{1}) $ with $ a(x) \geq 0 $. In the present paper, we naturally consider $ C^{1} $ regularity of solutions to the fully nonlinear non-local equation \eqref{Intro:eq1} with double degeneracy. To the best of our knowledge, few results were addressed on the borderline regularity of solutions to non-local equations. We now state our first main result, which establishes the borderline regularity of solutions to equation \eqref{Intro:eq1} in the degenerate case.

\begin{Theorem}\label{thm1}(borderline regularity in the degenerate case)
Let $ u \in C^{0}(\overline{B}_{1}) $ be a bounded viscosity solution to \eqref{Intro:eq1}, assume \hyperref[A1]{\bf (A1)}--\hyperref[A5]{\bf (A5)}, \hyperref[A6a]{\bf (A6a)} and \hyperref[A7]{\bf (A7)} hold. Then there exists $ \tau_{0} \in (1,2) $ sufficiently close to $ 2 $ such that if $ \tau \in (\tau_{0},2) $, then $ u \in C^{1}_{loc}(B_{1})$. Moreover, there exists a modulus of continuity $ \omega : \mathbb{R}_{0}^{+} \rightarrow  \mathbb{R}_{0}^{+}     $ depending only on $ n, \tau, \sigma_{1}, \sigma_{2}, ||u||_{L^{\infty}(B_{1})}, ||f||_{L^{\infty}(B_{1})} $ such that
\begin{equation*}
 |Du(x)-Du(y)| \leq \omega (|x-y|),
\end{equation*}
for every $ x,y \in B_{1/4} $. In particular, if $ \sigma_{i}^{-1}(i=1,2) $ behave like H\"{o}lder continuous functions near the origin, then the solutions are locally $ C^{1,\alpha} $, for some $ 0 < \alpha <1 $.
\end{Theorem}

We note that it suffices to require $ \sigma_{2}^{-1} $ to be a  Dini-continuous modulus of continuity, without similar assumption on $ \sigma_{1}^{-1} $. Indeed, $ \sigma_{1}^{-1} $ is automatically Dini-continuous since $ \sigma_{2}^{-1} $ is Dini-continuous and $ \sigma_{1} \geq \sigma_{2} $. It is worth noting that the result of Theorem \ref{thm1} embraces the main results of \cite{ART15, CL13, C21, DE21, PDM24, PDEE22, DR20} for fully nonlinear local/non-local equations with degenerate terms in a unified way, see Table \ref{Table1}, for the details.

\begin{table}[!htp]\label{Table1}
   \centering
   {
    \linespread{1.5}      \selectfont
   \begin{tabular}{||c|c|c||}
        \hline
        Model PDEs & Regularity of $ u $ & References  \\
        \hline
        \multirow{2}{*}{$ |Du|^{\gamma} F(D^{2}u) =f $} & $ C^{1,\alpha}_{loc}, \ 0< \alpha \leq \frac{1}{1+\gamma}$ &  \cite[Theorem 1]{CL13}  \\
        \cline{2-3}
         & $ C^{1,\alpha}_{loc}, \ \alpha = \min \big\{ \alpha_{0}^{-}, \frac{1}{1+\gamma}   \big\}$ &  \cite[Theorem 3.1]{ART15}  \\
        \hline
        \multirow{2}{*}{$ \big(|Du|^{\gamma_{1}} + a(x) |Du|^{\gamma_{2}} \big)F(D^{2}u) =f $} & $ C^{1,\alpha}_{loc}, \ 0< \alpha \leq \frac{1}{1+\gamma_{1}}$ & \cite[Theorem 1]{C21} \\
        \cline{2-3}
        &  $ C^{1,\alpha}_{loc}, \ \alpha = \min \big\{ \alpha_{0}^{-}, \frac{1}{1+\gamma_{1}}   \big\}$ & \cite[Theorem 1.1]{DR20}    \\
        \hline
         $ \frac{\varphi(|Du|)}{|Du|} F(D^{2}u) = f $ & $ C^{1}_{loc} $ &  \cite[Theorem 1.1]{BBO23}  \\
        \hline
        $ \Phi(x,|Du|) F(D^{2}u) = f $ &  $ C^{1,\alpha}_{loc}, \alpha = \min \big \{\alpha_{0}^{-},\frac{1}{1+s(\Phi)}    \big\} $  &  \cite[Theorem 1.1]{SSKS24}   \\
        \hline
        $ F(D^{2}u) + b(x)|\nabla u|^{\beta} = f $ & $ u \in C^{1,\alpha}_{loc},\alpha \in (0,1) $  &  \cite[Theorem 1.1]{BD15}    \\
        \hline
          $ |\nabla u|^{\gamma} F(D^{2}u) + b(x)|\nabla u|^{\beta} = f  $  &   $ u \in C^{1,\alpha}_{loc}, \alpha \in (0,1) $  &    \cite[Theorem 1.1]{BDL19}   \\
        \hline
        $ \sigma(|Du|) F(D^{2}u) =f $ & $ C^{1}_{loc} $  &  \cite[Theorem 1]{PDEE22}     \\
         \hline
                $
  \begin{cases}
  |Du|^{\gamma}(F(D^{2}u) + h(x) \cdot \nabla u  ) = f \ \   \text{in} \ \ B_{1}  \\
  u = \varphi  \qquad \qquad \qquad  \qquad \qquad \qquad \text{on} \ \ \partial B_{1}         \\
  \end{cases} $  &   $ u \in C^{1,\alpha}(\overline{B}_{1}), \alpha \in (0,1)  $          & 
  \cite[Theorem 1.1]{BD14}    \\
       \hline
        $ -|Du|^{\gamma} \mathcal{I}_{\tau}(u,x) = f $  & $ C^{1,\alpha}_{loc}, \ \alpha = \min \big\{ \alpha_{0}^{-}, \frac{\tau-1}{1+\gamma}   \big\} $ &  \cite[Theorem 1.2]{DE21}  \\
         \hline
         $ -\big(|Du|^{\gamma_{1}} + a(x) |Du|^{\gamma_{2}} \big)\mathcal{I}_{\tau}(u,x) =f $ & $ C^{1,\alpha}_{loc}, \ \alpha = \min \big\{ \alpha_{0}^{-}, \frac{\tau-1}{1+\gamma_{1}}   \big\} $  & \cite[Theorem 1.2]{PDM24} \\
         \hline
         $ - \big( \sigma_{1}(|Du|)  + a(x) \sigma_{2}(|Du|)    \big) \mathcal{I}_{\tau}(u,x) = f $  &  $ C^{1}_{loc} $  &  Theorem \ref{thm1}  \\
         \hline
   \end{tabular}
   }
\caption{Comparing Theorem \ref{thm1} with previous regularity results for related degenerate equations \cite{ART15, CL13, C21, DE21, PDM24, PDEE22, DR20}. Here $ \alpha_{0} $ is the optimal H\"{o}lder exponent for solutions to $ F$ harmonic function, $ \varphi $ is an increasing continuous function and the inverse of $ \varphi $ is Dini-continuous, please refer to \cite[(A2)]{BBO23} for details. In addition, please refer to \cite[(A2)]{SSKS24} for the conditions that $ \Phi $ and $ s(\Phi)$ satisfy.} 
\end{table}

We illustrate the strategy, which will be used in the proof of Theorem   \ref{thm1}. Note that obtaining $ C^{1} $ regularity of solutions for (\ref{Intro:eq1}) is showing that the graph of function $ u $ can be approximated by an affine function with an error bounded by $ Cr$ in any ball of radius $ r $. For this purpose, we first show that a viscosity solution $ u $ of the equation
\begin{equation}\label{Intro:eq3p}
- \bigg( \sigma_{1}(|Du+\xi|)  + a(x) \sigma_{2}(|Du+\xi|)  \bigg) \mathcal{I}_{\tau}(u,x) = f(x) \ \ \text{in} \ \ B_{1}, \ \ \ \xi \in \mathbb{R}^{n} \ \ \text{is any vector},
\end{equation}
is locally Lipschitz continuous, by using a doubling variable method and Lemma \ref{lem2.3}(see section \ref{Section 2}) for large/small scopes. Then, in the process of deriving an approximation lemma, concerning (\ref{Intro:eq1}) and (\ref{Intro:eq3p}) with the solutions to the homogeneous, uniformly elliptic problem $ \mathcal{I}_{\infty} = 0 $, the main difficulty is that the degenerate term $ \sigma_{1} + a(\cdot)\sigma_{2} $ lacks of {\it non-collapsing} property. To handle such a challenge, we combine the idea from \cite{PDEE22} and the recent work of \cite{PDM24} to establish a new {\it shored up}-continuous module sequence $ (\sigma_{1}^{n})_{n\in \mathbb{N}} $. This information is an important step towards the proof of Theorem \ref{thm1} based on iteration and convergence analysis, see Section \ref{subsec:3.3}.

\begin{remark}\label{intro:rk1}
Theorem \ref{thm1} is also extended to multi-phase degenerate non-local equation, i.e.,
\begin{equation*}
  - \bigg\{ \sigma_{0}(|Du|) + \sum_{i=1}^{k} a_{i}(x)\sigma_{i}(|Du|)                   \bigg\} \mathcal{I}_{\tau}(u,x) = f(x) \ \ \text{in} \ \ B_{1},
\end{equation*}
where additional conditions need to be imposed:
\begin{equation*}
  \sigma_{0}(t) \geq \sigma_{1}(t) \geq \cdots \geq \sigma_{k}(t), \ \forall \  t\in [0,1], \ \      \sigma_{0}(1) \geq \sigma_{1}(1) \geq \cdots \geq \sigma_{k}(1) \geq 1,
\end{equation*}
\begin{equation*}
  \lim_{t\rightarrow 0} \sigma_{i}(t) = 0, \ \ i= 0,1,\cdots, k, \ 0 \leq a_{j}(x) \in C^{0}(B_{1}), \  j=1,\cdots, k
\end{equation*}
and $ \sigma_{k}^{-1} $ satisfies Dini condition.
\end{remark}
\begin{remark}\label{intro:rk2}
The proof of Theorem \ref{thm1} is performed by approximation to the second-order equation associated to \eqref{Intro:eq1} as $ \tau \rightarrow 2^{-} $. In effect, by assumption \hyperref[A5]{\bf (A5)} and $ \tau \rightarrow 2^{-} $, this limiting problem takes the form
\begin{equation*}
  - \big( \sigma_{1}(|Du|)  + a(x) \sigma_{2}(|Du|)    \big) \inf_{i} \sup_{j} k_{ij} \text{Tr}(D^{2}u) = 0 \ \ \text{in} \ \ B_{1}.
\end{equation*}
The cancellation property of the viscosity solutions yields the simplified equation:
\begin{equation*}
  \inf_{i} \sup_{j} k_{ij} \text{Tr}(D^{2}u) = 0 \ \ \text{in} \ \ B_{1},
\end{equation*}
which serves as the cornerstone of our analysis. This limiting behavior crucially informs the proof strategy, with the complete technical details being developed in Proposition \ref{prop:3.2}.
\end{remark}

As we have described above, in the ground-breaking work \cite{CL13}, Imbert and Silvestre proved that the solutions to (\ref{Intro:eq1p}) with $ f \in L^{\infty}(B_{1}) $, are $ C^{1,\alpha}_{loc}$, for some $ 0< \alpha \leq \frac{1}{1+\gamma} $. This regularity is optimal, even if the source term $ f $ is H\"{o}lder continuous and $ F = \Delta $. In effect, we recall an important example from \cite{CL13}. For any $ \alpha \in (0,1) $, the function $ u(x)= |x|^{1+\alpha} $ satisfies
\begin{equation*}
  |Du|^{\gamma}\Delta u = C(n,\alpha)|x|^{\theta} \ \ \text{in} \ \ B_{1},
\end{equation*}
where $ \theta = (1+\alpha)(1+\gamma)- (2+\gamma) $. One can see that $ \theta \leq \gamma $. Furthermore, $ u \in C^{1, \frac{1+\theta}{1+\gamma}}$ at the origin ($ 0 $ is a critical point of $ u$). However, observing that for any given $ \theta > \gamma $, the function $ u(x) = |x_{n}|^{2+\frac{\theta-\gamma}{1+\gamma}} $ satisfies
\begin{equation*}
  |Du|^{\gamma}\Delta u = f(x) \ \ \text{in} \ \ B_{1},
\end{equation*}
where
\begin{equation*}
  f(x) = C(n,\theta)|x_{n}|^{\gamma+\frac{\gamma(\theta-\gamma)}{1+\gamma}+\frac{\theta-\gamma}{1+\gamma}} = C(n,\theta) |x_{n}|^{\theta} \leq C(n,\theta) |x|^{\theta}.
\end{equation*}
It can easily be seen that $ u $ is exactly $ C^{2, \frac{\theta-\gamma}{1+\gamma}} $ at the origin.

These examples suggest that the H\"{o}lder continuity of $ f $ has a direct impact on the regularity estimate of solutions to (\ref{Intro:eq1p}), and we can obtain higher regularity at certain meaningful points of solution $ u $ for (\ref{Intro:eq1p}) if $ f $ vanishes faster than degenerate law $ \gamma $. In a very recent paper \cite{T24}, Nascimento strictly proved Schauder-type regularity at interior critical points of solutions for (\ref{Intro:eq1p}). Similar results for fully nonlinear elliptic equation with general source term have also been discussed, and the reader is referred to \cite{TGE24}. Complementary to this, Da Silva et al. \cite{DDRS24} established higher regularity estimates for weak solutions of weighted quasilinear elliptic models of Hardy-H\'{e}non-type, in some specific scenarios. These advances collectively enriching the framework for analyzing degenerate elliptic operators.

In our work, we examine higher regularity of solutions for a class fully non-linear equation with double degeneracy of the form
\begin{equation}\label{intro:eq5}
- \bigg(|Du|^{\widetilde{p}}  + a(x) |Du|^{q}    \bigg) \mathcal{I}_{\tau}(u,x) = f(x) \ \ \text{in} \ \ B_{1},
\end{equation}
where $ 0< \widetilde{p} \leq q < \infty $, and $ \widetilde{p} < \tau -1$. Before proceeding, we give source term $ f $ a stronger assumption than \hyperref[A3]{\bf (A3)}, namely,

\vspace{2mm}

\label{A3a} {\bf (A3a)}: suppose that $ f(0)=0 $ and for some $ K >0 $ and $ \gamma \in (\widetilde{p}+2-\tau,1)$, there holds
\begin{equation*}
|f(x)| \leq K|x|^{\gamma}.
\end{equation*}

Now we state the second main result of this paper.

\begin{Theorem}\label{thm2}(higher regularity in the degenerate case)
Let $ u \in C^{0}(\overline{B}_{1}) $ be a bounded viscosity solution to (\ref{intro:eq5}), assume \hyperref[A1]{\bf (A1)}, \hyperref[A2]{\bf (A2)}, \hyperref[A3a]{\bf (A3a)}, \hyperref[A4]{\bf (A4)}and \hyperref[A5]{\bf (A5)} hold. Assume in addition that $ 0 $ is a local extrema for $ u $. Then, there exists $ \tau_{0} \in (1,2) $ sufficiently close to $ 2 $ such that if $ \tau \in (\tau_{0},2) $, then $ u \in C^{2,\alpha}(0) $ and
\begin{equation*}
  |u(x)-u(0)| \leq C|x|^{2+\alpha},
\end{equation*}
for any $ x \in B_{1/10}(0) $, where
\begin{equation*}
  \alpha = \frac{\gamma+ \tau -2-\widetilde{p}}{1+\widetilde{p}}\in (0,1),
\end{equation*}
and $ C = C(n,\gamma,\tau,\widetilde{p},\lambda,\Lambda)>0 $ is a positive constant.
\end{Theorem}

Observing the constant $ C $ in Theorem \ref{thm2} will not blow up as $ \tau \rightarrow 2 $. Hence, Theorem \ref{thm2} could be considered as an extension of \cite[Theorem 3]{T24}. The proof of Theorem \ref{thm2} relies on several ingredients. First, we can build a new improvement of flatness estimate at the local extrema, see Lemma \ref{se4:lemma1}. Then we utilize the flatness estimate and scaling techniques repeatedly to obtain the desired geometric decay.

\vspace{2mm}

In the singular case, we shall replace the condition \hyperref[A6a]{\bf (A6a)} by \hyperref[A6b]{\bf (A6b)}. This condition covers singular case, such as
\begin{equation*}
  \sigma_{1}(t) = t^{\widehat{p}}, \ \ \sigma_{2}(t) = t^{q}, \ -1 < \widehat{p} \leq q <0.
\end{equation*}
The condition \hyperref[A6b]{\bf (A6b)} also contains some important examples in \cite{SSKS24}, for instance,
\begin{equation*}
  \sigma_{1}(t) = \frac{(e^{t}-1)^{\beta}+1}{t}, \ \ \sigma_{2}(t) = \frac{(e^{t}-1)^{\beta}}{t}, \ \ 0 < \beta \leq 1-\frac{1}{e} ,
\end{equation*}
and
\begin{equation*}
  \sigma_{1}(t)= \frac{\ln^{\beta}(1+t)+1}{t}, \ \ \sigma_{2}(t)= \frac{\ln^{\beta}(1+t)}{t},\ \ 0 < \beta \leq 1 .
\end{equation*}

We state the third main result as follows.
\begin{Theorem}\label{thm3}($ C^{1,\beta_{0}} $ regularity in the singular case)
Let $ u \in C^{0}(\overline{B}_{1}) $ be a bounded approximated viscosity solution to \eqref{Intro:eq1}, assume \hyperref[A1]{\bf (A1)}, \hyperref[A2]{\bf (A2)}, \hyperref[A3]{\bf (A3)}, \hyperref[A4]{\bf (A4)} and \hyperref[A6b]{\bf (A6b)} hold. Then $ u \in C^{1,\beta_{0}}_{loc}(B_{1}) $. Moreover, there exists $ C >0 $ depending only on $ n, \lambda, \Lambda, ||f||_{L^{\infty}(B_{1})}$ and $||u||_{L^{\infty}(B_{1})}  $ such that
\begin{equation*}
  |Du(x)-Du(y)| \leq C|x-y|^{\beta_{0}},
\end{equation*}
for some $ \beta_{0} \in (0,1)$ and every $ x,y \in B_{1/2} $.
\end{Theorem}

For singular non-local equation, the challenge comes from the definition of viscosity solution, see the detailed description in \cite{PDM24}. Analogous to \cite[Theorem 1.3]{PDM24}, we shall use the notion of {\it approximated viscosity solution}. Due to the generality of $ \sigma_{i} $, $ i=1,2 $, Theorem   \ref{thm3} can be regarded as a slight generalization of \cite[Thorem 1.3]{PDM24}.

\vspace{2mm}

{\bf{Novelty of this paper.}} Here, we summarize the main novelties of this paper as follows:

\vspace{2mm}

$ \text{(i)} $ We combine ideas from Andrade {\it et.al.} \cite{PDEE22} and the recent work of Andrade-Prazeres-Santos \cite{PDM24} to address the obstruction caused by critical borderline regularity. To implement this strategy, we construct a completely new algorithm for selecting the normalization at each step, based on a {\it shoring-up} technique that effectively prevents the resulting collapse of $ \sigma_{1} + a(\cdot)\sigma_{2} $. Moreover, our result, Theorem \ref{thm1}, also extends the main results from \cite{ART15, CL13, C21, DE21, PDM24, PDEE22} in a unified way: see Table \ref{Table1}. This unified perspective is likely to be of independent interest.

\vspace{2mm}

$ \text{(ii)} $ We derive a non-local counterpart of \cite[Theorem 3]{T24} by employing a completely new flatness estimate at the local extrema, combined with scaling techniques. Thus, Theorem  \ref{thm2} generalizes the main result established in \cite[Theorem 3]{T24}.

\vspace{2mm}

$ \text{(iii)} $ On account of the generality of $ \sigma_{i}, i=1,2 $, \cite[Theorem 1.3]{PDM24} can be regarded as a direct consequence of Theorem   \ref{thm3}.

\vspace{2mm}

{\bf{State-of-the-art.}} In recent years, degenerate operator-driven equations have attracted growing research interest.

Regarding obstacle-type problems, Da Silva-Vivas in \cite{1DV21} and \cite{2DV21} established existence and optimal regularity estimates for degenerate elliptic models in non-divergence form:
\begin{equation*}
   \min \{f-|\nabla u|^{\gamma} \Delta u, u -\phi \} = 0 \ \ \text{in} \ \ B_{1} \ \ \left(\text{resp.} \ \ |\nabla u|^{\gamma} = 1\chi_{\{u>\phi\}}\right),
\end{equation*}
where $ \gamma > 0 $, $ f \in L^{\infty}(B_{1}) $, and $ \phi \in C^{1,\beta}(B_{1}) $ with $ \beta \in (0,1] $. Their work demonstrates that solutions belong to $ C^{1,\min\{\beta,\frac{1}{1+\gamma}\}} $, with \cite{2DV21} employing innovative geometric techniques to derive sharp regularity estimates for $ |\nabla u|^{\gamma} = 1\chi_{\{u>\phi\}} $.

Parallel developments include the work of Bezerra-J\'{u}nior et al.\cite{BdaR223, BdaRR23}, who established locally (uniformly) Lipschitz estimates for fully nonlinear singularly perturbed models with non-homogeneous degeneracy laws, along with global gradient estimates for a class of fully nonlinear PDEs featuring unbalanced variable degeneracy.

\vspace{1mm}

The field has witnessed particularly influential advances in free transmission problems. De Filippis' seminal work \cite{C22} pioneered the study of fully nonlinear elliptic equations with nonhomogeneous degeneracies, proving existence and local $ C^{1,\alpha} $ regularity for viscosity solutions to
\begin{align*}
 \left\{
     \begin{aligned}
     &  \big(|Du|^{p_{u}(x)} + a(x) \chi_{\{u>0\}} |Du|^{q} + b(x) \chi_{\{u<0\}} |Du|^{s}        \big) F(D^{2}u) = f(x)  \qquad  \ \ \text{in} \ \  B_{1}          \\
     & u=g  \qquad \qquad \qquad \qquad \qquad \qquad \qquad \qquad \qquad \qquad \qquad \qquad \qquad  \qquad \  \ \text{on}  \  \partial B_{1}                  ,          \\
     \end{aligned}
     \right.
\end{align*}
where $ p_{u}(x) = p^{+} \chi_{\{u>0\}} + p_{-} \chi_{\{u<0\}} $, $ 0 \leq p^{+} \leq q $, $ 0 \leq p_{-} \leq s  $ and $ 0 \leq a(\cdot), b(\cdot) \in C^{0}(B_{1}) $. Subsequently, Wang-Yin-Jiang in \cite{WYJ25} recently established analogous results for normalized $ p $-Laplacian equation with nonhomogeneous degeneracies, thereby broadening the applicability of regularity theory to a wider class of degenerate elliptic problems.

\vspace{1mm}

Finally, completing this state-of-the-art review, we must mention the recent work by Bezerra J\'{u}nior-da Silva in \cite{Bda25}, which combined the local estimates from \cite{BdaR223} with Harnack's inequality, A.B.P. Maximum Principle etc to establish global Lipschitz estimates for fully non-linear singular perturbation models with non-homogeneous signature. For a comprehensive overview of recent developments in fully nonlinear PDEs with unbalanced degeneracy laws, we refer the readers to \cite{BdaRRV22}.

\vspace{2mm}

{\bf{Organization of the paper.}} In Section \ref{Section 2}, we present some definitions and collect several auxiliary results. In Section \ref{Section 3}, we give a complete proof of Theorem \ref{thm1}. Section \ref{Section 4} is dedicated to the proof of Theorem \ref{thm2}. In Section \ref{Section 5}, we present the proof of Theorem \ref{thm3}. Finally, we finish the whole paper by proving borderline regularity of a degenerate non-local normalized $ p$-Laplacian equation.

\vspace{3mm}

\section{Preliminaries}\label{Section 2}
In this section, we present some definitions and show some preliminary results, which will be used throughout the paper.

\subsection{Notations}\label{subsec:2.1}
We first present the definition of viscosity solution to (\ref{Intro:eq1}).

\begin{Definition}
 We say that $ u \in C^{0}(\overline{B}_{1}) \cap L_{\tau}^{1}(\mathbb{R}^{n}) $ is a viscosity subsolution to (\ref{Intro:eq1}), if for any $ \varphi \in C^{2}(\mathbb{R}^{n}) $ and for all $ x_{0} \in B_{1} $ such that $ u-\varphi $ has a local maximum at $ x_{0} $, we have
 \begin{equation*}
   - \bigg( \sigma_{1}(|D\varphi|) +a(x_{0}) \sigma_{2}(|D\varphi|)    \bigg) \mathcal{I}_{\delta}(u,\varphi,x_{0}) \leq f(x_{0}) \ \ \text{in} \ \ B_{1},
 \end{equation*}
 where the operator $ \mathcal{I}_{\delta} $ is given by
 \begin{align*}
   \mathcal{I}_{\delta}(u,\varphi,x) &:= \inf_{i} \sup_{j} \big(I_{K_{ij}}[B_{\delta}](\varphi,x) + I_{K_{ij}}[B_{\delta}^{c}](u,x)    \big) \\
   & = \inf_{i} \sup_{j} \bigg[P.V. \int_{B_{\delta}} \big(\varphi(x+y)-\varphi(x)  \big)K_{ij}(y)dy + P.V. \int_{B_{\delta}^{c}} \big(u(x+y)-u(x)  \big)K_{ij}(y)dy                          \bigg].
 \end{align*}
 Similarly, we say that $ u \in C^{0}(\overline{B}_{1}) \cap L_{\tau}^{1}(\mathbb{R}^{n}) $ is a viscosity subsolution to (\ref{Intro:eq1}), if for any $ \varphi \in C^{2}(\mathbb{R}^{n}) $ and for all $ x_{0} \in B_{1} $ such that $ u-\varphi $ has a local minimum at $ x_{0} $, we have
  \begin{equation*}
   - \bigg( \sigma_{1}(|D\varphi|) +a(x_{0}) \sigma_{2}(|D\varphi|)    \bigg) \mathcal{I}_{\delta}(u,\varphi,x_{0}) \geq f(x_{0}) \ \ \text{in} \ \ B_{1}.
 \end{equation*}
 Finally, we say that $ u \in C(\mathbb{R}^{n}) $ is a viscosity solution to (\ref{Intro:eq1}) if it is both a viscosity subsolution and supersolution.
\end{Definition}

For singular fully nonlinear non-local elliptic equation, similar to \cite{PDM24}, we introduce the notion of {\it approximated viscosity solution}.

\begin{Definition}\label{def22}(approximated viscosity solution) Under the assumptions of Theorem \ref{thm3}, we say that $ u \in C^{0}(\overline{B}_{1}) $ is an approximated viscosity solution to
\begin{equation}\label{Pre:eq1}
- \bigg( \sigma_{1}(|Du|)  + a(x) \sigma_{2}(|Du|)    \bigg) \mathcal{I}_{\tau}(u,x) = f(x) \ \ \text{in} \ \ B_{1},
\end{equation}
if there are sequences $(u_{j})_{j\in\mathbb{N}} \in C^{0}(\overline{B}_{1})\cap L_{\tau}^{1}(\mathbb{R}^{n})  $, $ C_{1} > 0 $, $\alpha\in (-1, \tau-1)$ and $ (c_{j})_{j\in\mathbb{N}} \in \mathbb{R}^{+}  $ satisfying
\begin{enumerate}
    \item $ u_{j} \rightarrow u   $ locally uniformly in $ B_{1} $;

    \item $ ||u_{j}||_{L_{\tau}^{1}(\mathbb{R}^{n})}  \leq C_{1}(1+|x|^{1+\alpha}) $;

    \item $ c_{j} \rightarrow 0 $,
\end{enumerate}
such that $ u_{j} $ is a viscosity solution to
    \begin{equation*}
- \bigg( \sigma_{1}(|Du_{j}|+c_{j})  + a(x) \sigma_{2}(|Du_{j}|+c_{j})    \bigg) \mathcal{I}_{\tau}(u_{j},x) = f(x) \ \ \text{in} \ \ B_{1}.
    \end{equation*}
\end{Definition}

\begin{remark}\label{Pre:rk1}
Under the assumptions of Theorem \ref{thm3}, if $ u \in C^{0}(\overline{B}_{1}) $ is an approximated viscosity solution to equation (\ref{Pre:eq1}), this implies $ u $ is a viscosity solution to (\ref{Pre:eq1}), see, for instance, \cite[Proposition 3.1]{PDM24}. It is noteworthy that this conclusion also holds for non-local normalized $ p$-Laplacian operator $ \Delta_{p,N}^{s} $, see Section \ref{section 6}.
\end{remark}

To construct a family of non-collapsing moduli of continuity $\Sigma $ in Section \ref{subsec:3.3}, similar to \cite{PDEE22}, we also introduce the notation of {\it shored-up}.

\begin{Definition}\label{Def23}
A sequence of moduli of continuity $(\sigma_{1}^{k})_{k \in \mathbb{N}} $ is said to be shored-up if there exists a sequence of positive numbers $ (c_{k})_{k \in \mathbb{N}} $ such that $ c_{k} \rightarrow 0 $ satisfying $ \inf_{k} \sigma_{1}^{k}(c_{k}) > 0 $ for every $ k \in \mathbb{N} $.
\end{Definition}

\subsection{Auxiliary results}
In the study of interior regularity of (\ref{Intro:eq1}), a fundamental ingredient is the compactness for perturbed PDEs. Hence, we first recall the non-local version of Jensen-Ishii lemma, see \cite[Corollary 4.3.1]{GC08}, used to produce compactness in this paper.

\begin{Lemma}\label{lem2.3}
Let $ F : \mathbb{R}^{n} \times \mathbb{R} \times \mathbb{R}^{n} \times \mathbb{S}^{n} \times \mathbb{R}  \rightarrow \mathbb{R} $ be a continuous function satisfying the local and nonlocal degenerate ellipticity conditions: for any $ x \in \mathbb{R}^{n} $, $ u \in \mathbb{R} $, $ p \in \mathbb{R}^{n} $, $ M, N \in \mathbb{S}^{n}$, $ l_{1}, l_{2} \in \mathbb{R} $,
\begin{equation*}
  F(x,u,p,M,l_{1}) \leq F(x,u,p,N,l_{2}) \ \ \text{if} \ \ M \geq N, l_{1} \geq l_{2}.
\end{equation*}
Let $ u $ be a viscosity solution of
$$  F(x, u, Du, D^{2}u, \mathbf{I}[x,u])=0 \ \ \text{in} \ \ B_{1} .  $$
Consider a function $ v : B_{1} \times B_{1} \rightarrow \mathbb{R}   $ given by
$$  v(x,y):= u(x) - u(y)  .    $$
Let $ \varphi \in C^{2}(B_{1} \times B_{1})   $ and $ (x_{1}, y_{1}) \in B_{1} \times B_{1}   $ be a point of global maxima of $ v(x,y)- \varphi(x,y) $ in $ B_{1} \times B_{1} $. Then there exists $ \epsilon >0 $ small enough and matrices $ X,Y \in \mathbb{S}^{n}  $ such that
$$ F\big(x_{1}, u(x_{1}), D_{x}\varphi(\cdot,y_{1}),X, \mathbf{I}^{1,\delta}[x_{1},\varphi(\cdot,y_{1})]+ \mathbf{I}^{2,\delta}[x_{1},D_{x}\varphi(\cdot,y_{1}),u]  \big) \leq 0 ,    $$
and
$$ F\big(y_{1}, u(y_{1}), -D_{y}\varphi(x_{1},\cdot),Y , \mathbf{I}^{1,\delta}[y_{1},-\varphi(x_{1},\cdot)] + \mathbf{I}^{2,\delta}[y_{1},-D_{y}\varphi(x_{1},\cdot),u]  \big) \geq 0 ,     $$
where
$$ \mathbf{I}^{1,\delta}[x,\phi] = \int_{B_{\delta}} \bigg( \phi(x+y)-\phi(x)-D\phi(x) \cdot y \bigg)\mu (dy) ,$$
$$ \mathbf{I}^{2,\delta}[x,p,\omega] = \int_{\mathbb{R}^{n}\setminus B_{\delta}} \bigg( \omega(x+y)-\omega(x)-p \cdot y \mathbb{I}_{\{z\leq1\}}\bigg)\mu (dy) ,$$
for $ \delta \ll 1 $ and suitable measure $ \mu $. Also, we have the following matrix inequality
\begin{equation*}
-\frac{1}{\epsilon} \textbf{I} \leq
\begin{pmatrix}
X  &   0   \\
0  &    -Y
\end{pmatrix}
\leq
\begin{pmatrix}
D_{x}^{2}\varphi(x_{1},y_{1})   &   - D_{xy}^{2} \varphi(x_{1},y_{1})  \\
- D_{xy}^{2} \varphi(x_{1},y_{1})   &    D_{y}^{2}\varphi(x_{1},y_{1})
\end{pmatrix}
 + \epsilon\textbf{I} , \ \  0 < \epsilon \ll 1 .
\end{equation*}
\end{Lemma}

We continue with a technical lemma ensuring the existence of {\it shored up}-continuous module sequence $ (\sigma_{1}^{n})_{n\in \mathbb{N}} $, see \cite[Lemma 1]{PDEE22}; see also \cite[Proposition 3]{PS24}.

\begin{Lemma}\label{Au:prop1}
Let $ (a_{j})_{j\in\mathbb{N}} \in \ell^{1}  $ and take $ \epsilon, \delta >0 $ arbitrary. There exists a sequence $ (c_{j})_{j\in\mathbb{N}} \in c_{0}  $, with $ \max_{j\in\mathbb{N}}|c_{j}| \leq \epsilon^{-1}     $, such that
\begin{equation*}
\bigg( \frac{a_{j}}{c_{j}}    \bigg)_{j\in\mathbb{N}} \in \ell^{1},
\end{equation*}
and
\begin{equation*}
\epsilon \bigg( 1-\frac{\delta}{2}    \bigg)||(a_{j})||_{\ell^{1}} \leq \bigg|\bigg|\bigg( \frac{a_{j}}{c_{j}}    \bigg)\bigg|\bigg|_{\ell^{1}} \leq \epsilon(1+\delta)||(a_{j})||_{\ell^{1}}.
\end{equation*}
\end{Lemma}

Before proceeding, we shall provide a criterion of non-collapsing moduli of continuity $ \Sigma $, which will be used directly in Section \ref{subsec:3.3}, see \cite[Proposition 5]{PDEE22}.

\begin{Lemma}\label{Au:lem23}
  If a sequence of moduli of continuity $(\sigma_{1}^{k})_{k\in \mathbb{N}} $ is shored-up, then $ \Gamma := \cup_{k \in \mathbb{N}} \{\sigma_{1}^{k} \} $ is non-collapsing.
\end{Lemma}

Finally, we close this section with two vital remarks.

\begin{remark}\label{RK:22}(Scaling properties)
In the proof of Theorem \ref{thm1}, we require
\begin{equation}\label{sec23:rk26}
 ||u||_{L^{\infty}(B_{1})} \leq 1 \ \ \text{and} \ \ ||f||_{L^{\infty}(B_{1})} \leq  \epsilon ,
\end{equation}
where $ \epsilon >0 $ is to be determined. The condition in (\ref{sec23:rk26}) is not restrictive. In effect, by setting
\begin{equation*}
v(x) = \frac{u(rx)}{K},
\end{equation*}
for $ 0 < r \leq 1 $ and $ K > 0 $ to be determined, simple calculations yield that
\begin{equation*}
  - \bigg( \widetilde{\sigma}_{1}(|Dv|)  + \widetilde{a}(x) \widetilde{\sigma}_{2}(|Dv|)    \bigg) \widetilde{\mathcal{I}}_{\tau}(v,x) = \widetilde{f}(x) \ \ \text{in} \ \ B_{1},
\end{equation*}
where
\begin{equation*}
 \widetilde{\sigma}_{i}(t) = \sigma_{i}\bigg(\frac{K}{r}t\bigg), \ \ i=1,2 ; \ \ \widetilde{\mathcal{I}}_{\tau}(v,x) = \frac{r^{\tau}}{K} \mathcal{I}_{\tau}(Kv(x),rx);
\end{equation*}
\begin{equation*}
\widetilde{f}(x) = \frac{r^{\tau}}{K} f(rx); \ \ \widetilde{a}(x)= a(rx).
\end{equation*}
Notice that
$$ \widetilde{\sigma}_{i}^{-1}(t) = \frac{r}{K} \sigma_{i}^{-1}(t), \ \ i=1,2,  $$
then by choosing $ r < K $, it easily follows that
\begin{equation*}
  \int_{0}^{1} \frac{\widetilde{\sigma}_{2}^{-1}(t)}{t} dt \leq  \int_{0}^{1} \frac{\sigma_{2}^{-1}(t)}{t} dt  < \infty,  \ \ \widetilde{\sigma}_{1}(1)=\sigma_{1}\bigg(\frac{K}{r}\bigg) \geq \sigma_{1}(1) \geq \sigma_{2}(1) \geq 1,
\end{equation*}
and
\begin{equation*}
  0 \leq \widetilde{a}(x) \in C(B_{1}).
\end{equation*}
Hence $ v, \widetilde{f}, \widetilde{a}, \widetilde{\sigma}_{i}, i=1,2  $ meet assumptions \hyperref[A2]{\bf (A2)}, \hyperref[A3]{\bf (A3)}, \hyperref[A4]{\bf (A4)} and \hyperref[A6a]{\bf (A6a)}, \hyperref[A7]{\bf (A7)}, respectively. Clearly, $ \widetilde{I}_{\tau} $ satisfies \hyperref[A5]{\bf (A5)}. Finally, by choosing
\begin{equation*}
 r:= \epsilon \ \ \text{and} \ \ K:= ||u||_{L^{\infty}(B_{1})} + ||f||_{L^{\infty}(B_{1})},
\end{equation*}
we can assume (\ref{sec23:rk26}) without loss of generality.

We say that $ u \in C(B_{1}) $ is a normalized solution if $ ||u||_{L^{\infty}(B_{1})} \leq 1$ holds in the proof of Theorem \ref{thm1}.
\end{remark}

\begin{remark}\label{Rk26}
Dini condition can also be characterized in terms of the summability of $ \sigma_{2}^{-1} $ evaluated along geometric sequences, i.e., $ \sigma_{2}^{-1} $ satisfies the Dini condition if and only if
\begin{equation*}
\sum_{k=1}^{\infty} \sigma_{2}^{-1}(\theta^{k}) < \infty ,
\end{equation*}
for every $ \theta \in (0,1) $. This elementary quality is frequently used in the proof of Theorem \ref{thm1}.
\end{remark}

\vspace{3mm}

\section{Borderline regularity in the degenerate case}\label{Section 3}
In this section, we present the proof of Theorem \ref{thm1}. Our strategy consists of three steps. In Section \ref{The:3.1}, we prove a Lipschitz regularity result of solution to the perturbed equation using Lemma \ref{lem2.3} for large and small slopes. In Section \ref{The:3.2}, we establish an approximation result. In Section \ref{subsec:3.3}, we utilize approximation lemma to construct a sequence of approximating hyperplanes.                                                                                                      	
\subsection{Compactness of perturbed PDEs}\label{The:3.1}

\begin{Proposition}\label{Prop3.1}
Let $ u \in C^{0}(\overline{B}_{1}) $ be a normalized viscosity solution to
\begin{equation*}\label{The:eq1}
  - \bigg( \sigma_{1}(|Du+\xi|)  + a(x) \sigma_{2}(|Du+\xi|)  \bigg) \mathcal{I}_{\tau}(u,x) = f(x) \ \ \text{in} \ \ B_{1},
\end{equation*}
where $ \xi \in \mathbb{R}^{n}  $ is arbitrary, assume \hyperref[A1]{\bf (A1)}--\hyperref[A4]{\bf (A4)} and \hyperref[A6a]{\bf (A6a)} hold. Then $ u $ is locally Lipschitz continuous in $ B_{1} $, i.e., for every $ x,y \in B_{1/2} $,
\begin{equation*}
 |u(x)-u(y)| \leq C |x-y|,
\end{equation*}
where $ C > 0 $ is a universal constant.
\end{Proposition}

This proposition can be proved similar to \cite[Lemma 3.1]{PDM24} by using a doubling variable method and Lemma \ref{lem2.3}. For the sake of brevity, we omit its proof.

\subsection{Approximation lemma}\label{The:3.2}

\begin{Proposition}\label{prop:3.2}
Let $ u \in C^{0}(\overline{B}_{1}) $ be a normalized viscosity solution to
\begin{equation*}
  - \bigg( \sigma_{1}(|Du+\xi|)  + a(x) \sigma_{2}(|Du+\xi|)    \bigg) \mathcal{I}_{\tau}(u,x) = f(x) \ \ \text{in} \ \ B_{1},
\end{equation*}
where $ \xi \in \mathbb{R}^{n}  $ is arbitrary, assume \hyperref[A1]{\bf (A1)}--\hyperref[A6a]{\bf (A6a)} and \hyperref[A8]{\bf (A8)} hold. For any given $ \mathcal{M}, \delta >0 $, there exists $ \epsilon >0 $ such that if
\begin{equation*}
|u(x)| \leq \mathcal{M}(1+|x|^{1+\alpha_{0}}), \ \ \forall x \in \mathbb{R}^{n},
\end{equation*}
and
\begin{equation*}
|\tau -2| + ||f||_{L^{\infty}(B_{1})} \leq \epsilon,
\end{equation*}
then there exists a function $ h \in C^{1,\alpha_{0}}_{loc}(B_{1}) $ satisfying
\begin{equation*}
||u-h||_{L^{\infty}(B_{3/4})} \leq \delta,
\end{equation*}
where $ \alpha_{0} $ is the optimal H\"{o}lder exponent to the gradient for solutions of $ \mathcal{I}_{\infty} $ harmonic function.
\end{Proposition}

\begin{proof}
We argue by contradiction. For ease of presentation, we split the proof into five steps.

{\bf Step 1}. Suppose on the contrary that the result does not hold, then there exist the sequences $ \{\sigma_{1}^{k}\}_{k} $, $ \{\sigma_{2}^{k}\}_{k} $, $ \{u_{k}\}_{k} $, $ \{\tau_{k}\}_{k} $, $ \{\xi_{k}\}_{k} $, $ \{f_{k}\}_{k} $, $ \{a_{k}\}_{k}$, $ \{ \mathcal{I}_{\tau_{k}}\}_{k}$ and the numbers $ \mathcal{M}_{0}, \delta_{0} >0 $ such that

1. $ \{u_{k}\} $ is a normalized viscosity solution to
\begin{equation*}
 - \bigg( \sigma_{1}^{k}(|Du_{k}+\xi_{k}|)  + a_{k}(x) \sigma_{2}^{k}(|Du_{k}+\xi_{k}|)    \bigg) \mathcal{I}_{\tau_{k}}(u_{k},x) = f_{k}(x)  \ \ \text{in} \ \ B_{1};
\end{equation*}

2. \begin{equation*}
 |u_{k}(x)| \leq \mathcal{M}_{0}(1+|x|^{1+\alpha_{0}});
\end{equation*}

3. \begin{equation*}
 |\tau_{k}-2| + ||f_{k}||_{L^{\infty}(B_{1})} \leq \frac{1}{k};
\end{equation*}

4. The function $ \sigma_{i}^{k}(1) \geq 1  $, $ i=1,2 $, and if $ \sigma_{1}^{k}(b_{k}) + a_{k}(x) \sigma_{2}^{k}(b_{k})  \rightarrow 0   $, then $ b_{k} \rightarrow 0 $;

5. We have
\begin{equation}\label{se3:eq314}
||u_{k}-h||_{L^{\infty}(B_{3/4})} \geq   \delta_{0} >0,
\end{equation} for every $ k \in \mathbb{N} $ and any $ h \in C^{1,\alpha_{0}}_{loc}(B_{1}) $.

{\bf Step 2}. Since $ \tau_{k} \rightarrow 2 $ and \hyperref[A5]{\bf (A5)}, we have that $ \mathcal{I}_{\tau_{k}} \rightarrow \mathcal{I}_{\infty}   $, where $ \mathcal{I}_{\infty} $ is a uniformly elliptic operator. Moreover, in terms of Step 1 and Proposition \ref{Prop3.1}, there exists a subsequence, still denoted by $ \{u_{k}\}_{k\in \mathbb{N}}   $, converging uniformly to some $ u_{\infty} $ in $ B_{1/4} $. Now we want to show $ u_{\infty} $ solves
\begin{equation*}
  \mathcal{I}_{\infty}(D^{2}u_{\infty}) = 0,
\end{equation*}
in the viscosity sense. We only prove a supersolution property, because its subsolutions counterpart is entirely analogous. For the function $ p(x) $ defined by
 \begin{equation*}
   p(x):= \frac{1}{2}(x-y)^{T} M(x-y) + b\cdot (x-y) + u_{\infty}(y),
 \end{equation*}
we assume $ p $ touches $ u_{\infty} $ from below at $ y \in B_{1}$. Since $ u_{k} \rightarrow u_{\infty} $ in $ B_{1/4} $, we consider that $ p_{k} $ touches $ u_{k} $ from below at some $ x_{k} $ and $ x_{k} \rightarrow y $, where
\begin{equation*}
  p_{k}(x) := \frac{1}{2}(x-x_{k})^{T} M(x-x_{k}) + b\cdot (x-x_{k}) + u_{\infty}(x_{k}).
\end{equation*}
Then we have 
\begin{equation}\label{se3:eq34}
 - \bigg( \sigma_{1}^{k}(|b+\xi_{k}|)  + a_{k}(x_{k}) \sigma_{2}^{k}(|b+\xi_{k}|)    \bigg) \mathcal{I}_{\delta}(u_{k},p_{k},x_{k}) \geq f_{k}(x_{k}).
\end{equation}
The proof is completed if we verify $ \mathcal{I}_{\infty}(M) \leq 0 $.

{\bf Step 3}. Suppose that $ \{\xi_{k}\}_{k} $ does not admit convergence subsequence, namely, $ |\xi_{k}|\rightarrow \infty $ as $ k \rightarrow \infty $, then there exists a $ N > 0 $ such that
\begin{equation*}
  |b+\xi_{k}| \geq 1,
\end{equation*}
as $ k > N $. From \hyperref[A4]{\bf (A4)}, \hyperref[A6a]{\bf (A6a)} and (\ref{se3:eq34}), we obtain 
\begin{equation*}
 -\mathcal{I}_{\delta}(u_{k},p_{k},x_{k}) \geq - \frac{|f_{k}(x_{k})|}{\sigma_{1}^{k}(|b+\xi_{k}|)} \geq -\frac{1}{k}.
\end{equation*}
By letting $ k \rightarrow \infty $, we get $ \mathcal{I}_{\infty}(M) \leq 0 $.

{\bf Step 4}. Suppose that $ \{\xi_{k}\}_{k} $ admits convergence subsequence. Without loss of generality, there exists $ \xi_{0} \in \mathbb{R}^{n}  $ such that $ \xi_{k} \rightarrow  \xi_{0}   $. Then one can find $ N >0 $ such that
\begin{equation*}
  |b+\xi_{k}| \geq \frac{1}{2} |b+\xi_{0}|>0,
\end{equation*}
for $ k > N $. By \hyperref[A8]{\bf (A8)}, we know $ \sigma_{1}(|b+\xi_{k}|) +a_{k}(x)\sigma_{2}(|b+\xi_{k}|) \nrightarrow 0 $. Then using \hyperref[A4]{\bf (A4)}, \hyperref[A6a]{\bf (A6a)} and (\ref{se3:eq34}) again, we get 
\begin{equation*}
 -\mathcal{I}_{\delta}(u_{k},p_{k},x_{k})  \geq -\frac{1}{k\sigma_{1}^{k}(\frac{1}{2}|b+\xi_{0}|)}.
\end{equation*}
By letting $ k \rightarrow \infty $, we get $ \mathcal{I}_{\infty}(M) \leq 0 $ again. For the remaining case $ |b+\xi_{0}| =0  $, the proof is analogous to     \cite[Lemma 3.5]{PDM24} and \cite[Proposition 9]{PS24}, we omit it here.

{\bf Step 5}. Since $ \mathcal{I}_{\infty}(D^{2}u_{\infty}) = 0 $, then as a consequence of Krylov-Safonov regularity theory for uniformly elliptic operator, see \cite[Corollary 5.7]{CC95}, we obtain that $ u_{\infty}\in C^{1,\alpha_{0}}_{loc}(B_{1}) $. By taking $ h = u_{\infty} $, this contradicts with (\ref{se3:eq314}).
We complete the proof of Proposition \ref{prop:3.2}.
\end{proof}

\subsection{Existence of approximating hyperplanes}\label{subsec:3.3}

We shall define appropriate moduli of continuity to ensure $ \mu_{1} > r^{\tau-1}$. We start by introducing $ \gamma (t)= t \sigma_{1}(t)  $. Because $ t \mapsto t \sigma_{1}(t)  $ is a bijective map, it has an inverse. Let $ \omega (t) = \gamma^{-1}(t)   $. We examine the choice of $ \mu_{1} $ by $ \omega (t) $ as follows.

Because $ \tau $ is sufficiently close to $ 2 $, this allows us to examine the case $ \tau > 1 + \alpha_{0} $.
Suppose first
 $$ \frac{t^{\alpha_{0}}}{\omega(t)} \rightarrow 0, $$
then choose small $ 0 < r < \frac{1}{2} $ such that
\begin{equation*}
2Lr^{\alpha_{0}} := \mu_{1} > r^{\tau -1}.
\end{equation*}
On the contrary, suppose
\begin{equation*}
  \frac{t^{\alpha_{0}}}{\omega(t)} \rightarrow M_{1},
\end{equation*}
for some constant $ M_{1} >0 $, then we select small $ 0 < r < \frac{1}{2} $ and any $ 0<\alpha < \alpha_{0} $ such that
\begin{equation*}
  2Lr^{\alpha_{0}} = r^{\alpha} := \mu_{1} > r^{\tau-1},
\end{equation*}
where we have used $ 0< \alpha < \alpha_{0} < \tau -1  $.

In what follows, we proceed by setting $  0< \theta :=   \frac{r^{\tau-1}}{\mu_{1}} <1   $ and considering the sequence
\begin{equation*}
(a_{k})_{k\in \mathbb{N}} := (\sigma_{1}^{-1}(\theta^{k}))_{k\in\mathbb{N}}.
\end{equation*}
Under the assumptions \hyperref[A6a]{\bf (A6a)} and \hyperref[A7]{\bf (A7)}, the inverse $ \sigma_{1}^{-1} $ is Dini-continuous. By Remark \ref{Rk26}, the sequence $(\sigma_{1}^{-1}(\theta^{k}))_{k\in\mathbb{N}} $ is summable, and $ (a_{k})_{k\in \mathbb{N}} \in \ell^{1} $.
Now we utilize Lemma \ref{Au:prop1}, there exists a sequence $ (c_{k})_{k\in \mathbb{N}} \in c_{0}$ such that
\begin{equation}
\frac{7}{10} \sum_{k=1}^{\infty} \sigma_{1}^{-1}(\theta^{k}) \leq \sum_{k=1}^{\infty} \frac{\sigma_{1}^{-1}(\theta^{k})}{c_{k}}  \leq \sum_{k=1}^{\infty} \sigma_{1}^{-1}(\theta^{k}).
\end{equation}

Finally, we design a sequence of moduli of continuity $ (\sigma_{1}^{k}(\cdot,x))_{k\in \mathbb{N}}  $ given by
\begin{equation}\label{loc:eq7}
\left\{
     \begin{aligned}
     &  \sigma_{1}^{0}(t,x) := \sigma_{1}(t) + a(x) \sigma_{2}(t) ,        \\
     &   \sigma_{1}^{1}(t,x) := \frac{\mu_{1}}{r^{\tau -1}} \bigg [ \sigma_{1}(\mu_{1}t) + a(rx) \sigma_{2}(\mu_{1}t) \bigg ] ,          \\
     & \sigma_{1}^{2}(t,x) :=  \frac{\mu_{1}\mu_{2}}{r^{2(\tau -1)}} \bigg [ \sigma_{1}(\mu_{1}\mu_{2}t) + a(r^{2}x) \sigma_{2}(\mu_{1}\mu_{2}t) \bigg ]           ,  \\
     & \cdots\cdots \cdots  \\
     &  \sigma_{1}^{k}(t,x) := \frac{\mu_{1}\mu_{2}\cdots \mu_{k}}{r^{k(\tau -1)}} \bigg [ \sigma_{1}(\mu_{1}\mu_{2}\cdots \mu_{k}t) + a(r^{k}x) \sigma_{2}(\mu_{1}\mu_{2}\cdots \mu_{k}t) \bigg ],
     \end{aligned}
     \right.
\end{equation}
with $ \mu_{1} > r^{\tau -1}    $ as defined and $ (\mu_{k})_{k \in \mathbb{N}}  $ determined as follows. If
\begin{equation*}
\frac{\prod_{i=1}^{k}\mu_{i}}{r^{k(\tau -1)}}\sigma_{1} \bigg(\prod_{i=1}^{k}\mu_{i} c_{k} \bigg) \geq 1,
\end{equation*}
then $ \mu_{k} = \mu_{k-1}    $. Otherwise $ \mu_{k-1} < \mu_{k} <1   $ is chosen to ensure
\begin{equation*}
\frac{\prod_{i=1}^{k}\mu_{i}}{r^{k(\tau -1)}}\sigma_{1} \bigg(\prod_{i=1}^{k}\mu_{i} c_{k} \bigg) = 1.
\end{equation*}

It can easily be seen that a sequence of moduli of continuity $ (\sigma_{1}^{k}(t,x))_{k \in \mathbb{N}} $ is {\it shored-up}, and using Lemma \ref{Au:lem23}, it follows that
$$ \Sigma = \big\{\sigma_{1}^{0}(t,x), \sigma_{1}^{1}(t,x), \cdots, \sigma_{1}^{k}(t,x), \cdots\big\}     $$
is non-collapsing.

\begin{Proposition}\label{prop3.3}
Let $ u \in C^{0}(\overline{B}_{1}) $ be a normalized viscosity solution to
\begin{equation*}
  - \bigg( \sigma_{1}(|Du+\xi|)  + a(x) \sigma_{2}(|Du+\xi|)    \bigg) \mathcal{I}_{\tau }(u,x) = f(x) \ \ \text{in} \ \ B_{1},
\end{equation*}
where $ \xi \in \mathbb{R}^{n}  $ is arbitrary, assume \hyperref[A1]{\bf (A1)}--\hyperref[A6a]{\bf (A6a)} and \hyperref[A7]{\bf (A7)} hold. Given $ \mathcal{M}, \delta >0 $, there exists $ \epsilon >0 $ such that if
\begin{equation*}
|u(x)| \leq \mathcal{M}(1+|x|^{1+\alpha_{0}}), \ \ \forall x \in \mathbb{R}^{n},
\end{equation*}
and
\begin{equation*}
|\tau -2| + ||f||_{L^{\infty}(B_{1})} \leq \epsilon,
\end{equation*}
then there exist an affine function $ l(x) = a + b \cdot x $ and a universal constant $ C >0 $ such that
\begin{equation*}
 |a|+|b|\leq C, \ \ \text{and} \ \ \sup_{x \in B_{r}} |u(x)-l(x)|\leq \mu_{1} r.
\end{equation*}
\end{Proposition}
\begin{proof}
 From Proposition \ref{prop:3.2}, there exists a function $ h \in C^{1,\alpha_{0}}_{loc}(B_{1}) $ such that
 \begin{equation}\label{sec3:eq38}
 \sup_{x\in B_{3/4}} |u(x)-h(x)| \leq \delta,
 \end{equation}
for some $ \delta > 0 $ to be determined. In addition, the regularity of function $ h $ yields that
\begin{equation}\label{sec3:eq39}
 \sup_{x \in B_{r}} |h(x)-h(0)-Dh(0)\cdot x| \leq Lr^{1+\alpha_{0}},
\end{equation}
for a universal constant $ L > 0 $ and every $ 0 < r <1 $. We combine (\ref{sec3:eq38}), (\ref{sec3:eq39}) and the choice of $ \mu_{1} $ to get
\begin{align*}
  \sup_{x \in B_{r}} |u(x)-a-b\cdot x|  & \leq \sup_{x \in B_{r}} |u(x)-h(x)| + \sup_{x \in B_{r}} |h(x)-a- b\cdot x|  \\
  & \leq \delta + Lr^{1+\alpha_{0}}  \leq  \delta + \frac{1}{2}\mu_{1} r,
\end{align*}
where $ a := h(0) $ and $ b := Dh(0) $. Now choosing $ \delta = \frac{1}{2}\mu_{1} r  $, there holds
\begin{equation*}
  \sup_{x \in B_{r}} |u(x)-a-b\cdot x| \leq \mu_{1} r.
\end{equation*}
Hence the proof is completed.
\end{proof}

\begin{Proposition}\label{Prop:34}
Let $ u \in C^{0}(\overline{B}_{1}) $ be a normalized viscosity solution to (\ref{Intro:eq1}), assume \hyperref[A1]{\bf (A1)}--\hyperref[A6a]{\bf (A6a)} and \hyperref[A7]{\bf (A7)} hold. Given $ \mathcal{M}, \delta >0 $, there exists $ \epsilon >0 $ such that if
\begin{equation*}
|u(x)| \leq \mathcal{M}(1+|x|^{1+\alpha_{0}}), \ \ \forall x \in \mathbb{R}^{n},
\end{equation*}
and
\begin{equation*}
|\tau -2| + ||f||_{L^{\infty}(B_{1})} \leq \epsilon,
\end{equation*}
then there exists an affine function $ (l_{k})_{k\in \mathbb{N}}  $ of the form
\begin{equation*}
  l_{k}(x)= a_{k} + b_{k} \cdot x
\end{equation*}
satisfying
\begin{equation*}
\sup_{x \in B_{r^{k}}} |u(x)-l_{k}(x)| \leq \bigg(\prod_{i=1}^{k}\mu_{i} \bigg) r^{k},
\end{equation*}
\begin{equation*}
|a_{k+1}-a_{k}| \leq C \bigg(\prod_{i=1}^{k}\mu_{i} \bigg) r^{k},
\end{equation*}
and
\begin{equation*}
|b_{k+1}-b_{k}| \leq C \bigg(\prod_{i=1}^{k}\mu_{i} \bigg).
\end{equation*}
\end{Proposition}
\begin{proof}
Let
\begin{equation*}
u_{k}(x) = \frac{u_{k-1}(rx)-l_{k-1}(rx)}{\mu_{k}r},
\end{equation*}
then by induction and simple calculation, we have that $ u_{k} $ satisfies
\begin{equation*}
  \sigma_{1}^{k}\bigg(|Du_{k}(x)+\frac{1}{\mu_{k}}b_{k-1}+\cdots +\frac{1}{\mu_{1}}b_{0}|,x\bigg) \mathcal{I}_{\tau}^{k}(x,u_{k}) = f_{k}(x) \ \ \text{in} \ \ B_{1},
\end{equation*}
where
\begin{align*}
  & \sigma_{1}^{k}(t,x) = \frac{\mu_{1}\mu_{2}\cdots \mu_{k}}{r^{k(\tau -1)}}\bigg [ \sigma_{1} (\mu_{1}\mu_{2}\cdots \mu_{k} t) + a(r^{k}x)\sigma_{2} (\mu_{1}\mu_{2}\cdots \mu_{k} t)    \bigg],  \\
 & \mathcal{I}_{\tau}^{k}(x,u_{k}) =  \frac{r^{k(\tau-1)}}{\mu_{1}\mu_{2}\cdots \mu_{k}} \mathcal{I}_{\tau }(r^{k}x, \mu_{1}\mu_{2}\cdots \mu_{k}r+\cdots),
\end{align*}
and
\begin{equation*}
  f_{k}(x)= f(r^{k}x) .
\end{equation*}
Now we want to show that
\begin{equation}\label{the:eq3.9}
|u_{k}(x)| \leq \mathcal{M}(1+|x|^{1+ \alpha_{0}}), \ \ \forall \  x \in \mathbb{R}^{n}.
\end{equation}
We again use an induction argument. For the case $ k = 0 $, we take $ u_{0} = u $. Suppose (\ref{the:eq3.9}) is valid for the case $ k=0, 1, \cdots, N-1 $, then we shall prove that (\ref{the:eq3.9}) is true in the case $ k=N $.

Case 1. $ |x|r > 1/2   $. By the induction hypothesis and $ 2Lr^{\alpha_{0}} := \mu_{1} < \mu_{N} < 1 $, we arrive
\begin{align*}
|u_{N}(x)| & \leq (\mu_{N}r)^{-1} \left(|u_{N-1}(rx)|+ |l_{N-1}(rx)|    \right)  \\
& \leq (\mu_{N}r)^{-1} \big[m(1+|rx|^{1+\alpha_{0}})\big] + (\mu_{N}r)^{-1} L(1+r|x|)  \\
& \leq \frac{2^{\alpha_{0}}m}{L}|x|^{1+\alpha_{0}} + \frac{m}{2L}|x|^{1+\alpha_{0}} + 2^{\alpha_{0}} |x|^{1+\alpha_{0}} + 2^{\alpha_{0}-1} |x|^{1+\alpha_{0}} : = K |x|^{1+\alpha_{0}},
\end{align*}
where $ K = \frac{2^{\alpha_{0}}m}{L} +\frac{m}{2L} + 2^{\alpha_{0}} + 2^{\alpha_{0}-1} $. By choosing a suitable $ \mathcal{M} $ such that $ K < \mathcal{M} $, then (\ref{the:eq3.9}) is verified.

Case 2. $ |x|r \leq 1/2   $. Using $ 2Lr^{\alpha_{0}} := \mu_{1} < \mu_{N} < 1 $ again, it follows that
\begin{align*}
|u_{N}(x)| & \leq (\mu_{N}r)^{-1} \bigg(|u_{N-1}(rx)-h(rx)| + |h(rx)-l_{N-1}(rx)|    \bigg)  \\
& \leq (\mu_{N}r)^{-1} \bigg( \frac{1}{2}\mu_{1}r  + Lr^{1+\alpha_{0}}|x|^{1+\alpha_{0}}      \bigg)   \\
&  \leq \frac{1}{2} + \frac{1}{2}|x|^{1+\alpha_{0}} \leq 1+ |x|^{1+\alpha_{0}},
\end{align*}
where $ h $ comes from Proposition \ref{prop3.3}. This finishes the proof of (\ref{the:eq3.9}).

Once we have verified (\ref{the:eq3.9}), we can apply Proposition \ref{prop3.3} to obtain
\begin{equation*}
  \sup_{x \in B_{r}} |u_{k}(x)-l_{k}(x)| \leq \mu_{1} r.
\end{equation*}
Scaling back to $ u $ yields
\begin{equation*}
  \sup_{x \in B_{r^{k+1}}} \bigg|u(x)- l_{k+1}(x) \bigg| \leq \mu_{1}^{2} \mu_{2} \cdots \mu_{k} r^{k+1} \leq \bigg( \prod_{i=1}^{k+1} \mu_{i}   \bigg)r^{k+1},
\end{equation*}
where
\begin{equation*}
  l_{k+1}(x) = l_{0}(x) + \sum_{i=1}^{k} \big( \prod_{j=1}^{i} \mu_{j}  \big) r^{i} l_{i}(r^{-i}x).
\end{equation*}
Moreover, we have
\begin{equation*}
  |a_{k+1}-a_{k}| = \bigg| \bigg(\prod_{i=1}^{k} \mu_{i}\bigg)r^{k}a_{k}\bigg| \leq C \bigg(\prod_{i=1}^{k} \mu_{i}    \bigg)r^{k},
\end{equation*}
and
\begin{equation*}
  |b_{k+1}-b_{k}| = \bigg|\bigg( \prod_{i=1}^{k} \mu_{i}\bigg)b_{k}   \bigg| \leq C \bigg(\prod_{i=1}^{k} \mu_{i}    \bigg).
\end{equation*}
This completes the proof.
\end{proof}

Now we are ready to present the proof of Theorem~\ref{thm1}.

\begin{proof}[Proof of Theorem~\ref{thm1}]
We mainly examine the convergence of the sequence of $ \prod_{i=1}^{k} \mu_{i}   $. Noting that two possibilities concerning the sequence $ \prod_{i=1}^{k} \mu_{i}   $ will happen. Either the sequence repeats after some index $ N \geq 2 $ or we have $ \mu_{n} < \mu_{n+1} $ for infinitely many indices $ n \in \mathbb{N} $.

In the formal case, it is well-known that $ C^{1,\alpha} $ regularity estimate are available, for $ 0 < \alpha < \alpha_{0} $, where $ \alpha_{0} \in (0,1) $ is the exponent associated with the regularity of $ I_{\infty} $ harmonic function. In the latter case, there holds
\begin{equation*}
\frac{\prod_{i=1}^{k}\mu_{i}}{r^{k(\tau -1)}}\sigma_{1} \bigg(\prod_{i=1}^{k}\mu_{i} c_{k} \bigg) = 1,
\end{equation*}
which implies that
\begin{equation*}
   \gamma_{k}:= \prod_{i=1}^{k} \mu_{i}  =  \frac{1}{c_{k}} \sigma_{1}^{-1} \bigg(\frac{r^{k(\tau-1)}}{\prod_{i=1}^{k}} \mu_{i}   \bigg) \leq \frac{\sigma_{1}^{-1}(\theta^{k})}{c_{k}}.
\end{equation*}
We apply Lemma \ref{Au:prop1} to get $ (\gamma_{k})_{k\in \mathbb{N}} \in \ell^{1} $, and its norm can be bounded by $ \sum_{k=1}^{\infty} \sigma_{1}^{-1}(\theta^{k}) $. Consequently, we deduce that
\begin{equation*}
  \lim_{k\rightarrow \infty} \bigg( \prod_{i=1}^{k} \mu_{i} \bigg) =0.
\end{equation*}
Thereby, $ (a_{k})_{k\in \mathbb{N}} $ and $ (b_{k})_{k\in \mathbb{N}} $ are Cauchy sequences and there exists $ a_{\infty} \in \mathbb{R} $ and $ b_{\infty} \in \mathbb{R}^{n} $ such that
\begin{equation*}
  a_{k} \rightarrow a_{\infty} \ \ \text{and} \ \ b_{k} \rightarrow b_{\infty},
\end{equation*}
as $ k \rightarrow \infty $. Moreover, there holds
\begin{equation}\label{Sec3:eq38}
 |a_{k}-a_{\infty}| \leq C \bigg( \sum_{i=k}^{\infty} \gamma_{i}\bigg) r^{k} \ \ \text{and} \ \  |b_{k}-b_{\infty}| \leq C \bigg( \sum_{i=k}^{\infty} \gamma_{i}  \bigg).
\end{equation}
Set $ l_{\infty}(x):= a_{\infty} + b_{\infty}\cdot x    $, and for any $ 0< \rho \ll 1$, then there exists $ k \in \mathbb{N} $ such that $ r^{k+1} < \rho \leq r^{k} $.
We combine proposition \ref{Prop:34} and \eqref{Sec3:eq38} to obtain
\begin{align}\label{Sec3:eq39}
\begin{split}
 \sup_{x\in B_{\rho}} |u(x)-l_{\infty}(x)| & \leq \sup_{x\in B_{r^{k}}} |u(x)-l_{k}(x)|  + \sup_{x\in B_{r^{k}}}  |l_{k}(x)-l_{\infty}(x)|   \\
& \leq C \gamma_{k} r^{k}  +  C \bigg( \sum_{i=k}^{\infty} \gamma_{i}\bigg) r^{k}  \\
& \leq C \bigg( \sum_{i=k}^{\infty} \gamma_{i}\bigg) \rho .
\end{split}
\end{align}
In the end, set
\begin{equation*}
    \Gamma(t) :=   \bigg( \sum_{i=\lfloor\ln t^{-1}\rfloor}^{\infty} \gamma_{i}\bigg),
\end{equation*}
where $ \lfloor\ln t^{-1}\rfloor $ denotes the biggest integer that is less than or equal to $ \ln t^{-1} $. Since $ \gamma_{i} \in \ell^{1} $, then $ \Gamma(t)  $ indeed is a modulus of continuity. Hence \eqref{Sec3:eq39} becomes
\begin{equation*}
  \sup_{x\in B_{\rho}} |u(x)-l_{\infty}(x)|  \leq C \Gamma(\rho) \rho,
\end{equation*}
which finishes the proof of Theorem \ref{thm1}.
\end{proof}

\vspace{3mm}

\section{$ C^{2,\alpha}$ regularity at the local extrema in the degenerate case}\label{Section 4}
In this section, we prove higher regularity of solutions to (\ref{intro:eq5}) at the local extrema point. Our proof consists of two steps. First, we establish a new flatness improvement estimate at the local extrema, see Lemma \ref{se4:lemma1}. Then we utilize Lemma \ref{se4:lemma1} and scaling techniques repeatedly to obtain geometric decay estimate.

Now we give a precise statement about flatness estimate.

\begin{Lemma}\label{se4:lemma1}(flatness estimate)
Let $ u \in C^{0}(\overline{B}_{1}) $ be a normalized viscosity solution to
\begin{equation*}
  - \bigg(|Du|^{\widetilde{p}}  + a(x) |Du|^{q}    \bigg) \mathcal{I}_{\tau}(u,x) = f(x) \ \ \text{in} \ \ B_{1},
\end{equation*}
assume \hyperref[A1]{\bf (A1)}--\hyperref[A5]{\bf (A5)} hold and assume also $ x_{0} \in B_{1/2} $ is a local minimum. Given $ \mathcal{M}, \delta >0 $, there exists $ \epsilon >0 $ such that if
\begin{equation*}
|u(x)| \leq \mathcal{M}(1+|x|^{1+\alpha_{0}}), \ \ \forall \ x \in \mathbb{R}^{n}
\end{equation*}
and
\begin{equation}\label{sec4:eq1}
|\tau-2| + ||f||_{L^{\infty}(B_{1})} \leq \epsilon,
\end{equation}
then
\begin{equation}\label{sec4:eq2}
\sup_{B_{1/10}(x_{0})} (u-u(x_{0})) \leq \delta .
\end{equation}

\end{Lemma}

\begin{proof}
Suppose, by contradiction, that there exist $ \mathcal{M}_{0}, \delta_{0} >0 $ and some sequences $ \{ u_{k}\}_{k}, \{ f_{k}\}_{k}, $ $ \{ a_{k}\}_{k}, \{ \mathcal{I}_{\tau_{k}}\}_{k} $, $    \{\tau_{k}\}_{k} $, $ x_{k} \in B_{1/2} $ such that the following conclusions hold:

\text{(1a)}. $ ||u_{k}||_{L^{\infty}(B_{1})} \leq 1 $ and $ |\tau_{k}-2| + ||f_{k}||_{L^{\infty}(B_{1})} \leq \frac{1}{k}    $ ;

\text{(1b)}. $ |u_{k}(x)| \leq \mathcal{M}_{0}(1+|x|^{1+\alpha_{0}}) $ ;

\text{(1c)}. $ - \bigg(|Du_{k}|^{\widetilde{p}}  + a_{k}(x) |Du_{k}|^{q}    \bigg) \mathcal{I}_{\tau_{k}}(u_{k},x) = f_{k}(x) \ \ \text{in} \ \ B_{1} $, in the viscosity sense.

However,
\begin{equation}\label{sec4:eq3}
\sup_{B_{1/10}(x_{0})} (u_{k}-u_{k}(x_{k})) \geq \delta_{0} >0 .
\end{equation}

From \cite[ Theorem 1.2]{PDM24} and the  Arzel$ \grave{a}$-Ascoli Theorem, it follows that
\begin{equation}\label{sec4:eq4}
 u_{k} \rightarrow u_{\infty} \ \text{uniformly in} \ \  B_{r} \ \text{and} \   Du_{k} \rightarrow Du_{\infty}\ \  \text{uniformly in} \ \  B_{r}.
\end{equation}
Since $ \tau_{k} \rightarrow 2   $ and \hyperref[A5]{\bf (A5)}, we have $ \mathcal{I}_{\tau_{k}} \rightarrow F_{\infty}   $, where $ F_{\infty} $ is a uniformly elliptic operator. In addition, using \hyperref[A4]{\bf (A4)}, it yields that
\begin{equation}\label{sec4:eq5}
\left\{
     \begin{aligned}
     &  |Du_{k}|^{\widetilde{p}} |\mathcal{I}_{\tau_{k}}| \leq |f_{k}(x)|,    \\
     &   |Du_{k}|^{\widetilde{p}} |\mathcal{I}_{\tau_{k}}| \geq -|f_{k}(x)|.         \\
     \end{aligned}
     \right.
\end{equation}
Taking $ k \rightarrow \infty $ in (\ref{sec4:eq5}), we have that $ u_{\infty} $ is a viscosity solution to
\begin{equation*}
  |Du_{\infty}|^{\widetilde{p}} F_{\infty}(D^{2}u_{\infty}) = 0 \ \ \text{in} \ \ B_{1/4},
\end{equation*}
which leads to
\begin{equation}
  F_{\infty}(D^{2}u_{\infty}) = 0 \ \ \text{in} \ \ B_{1/4},
\end{equation}
where we have used \cite[ Lemma 6]{CL13}.

Furthermore, $ x_{k} \rightarrow x_{\infty} $, and by the uniform convergence, $ x_{\infty} $ is a local minimum of $ u_{\infty} $. Consequently, by the strong maximum principle \cite[Proposition 4.9]{CC95}, $ u_{\infty} $ is a constant. Hence, we get a contradiction with (\ref{sec4:eq3}) for $ k $ sufficiently large.
\end{proof}

With the help of Lemma \ref{se4:lemma1}, we now show higher regularity of solution to (\ref{intro:eq5}).

\begin{proof}[Proof of Theorem~\ref{thm2}]
For simplicity, we assume $ u(0) =0  $ and $ 0 $ is a local minimum of $ u$. For $ 0 < \rho_{*} <1/10 $ to be chosen later. Defining $ v_{1}(x) = u(\rho_{*}x) $, then $ v_{1} $ solves
\begin{equation*}
  - \bigg(|Dv_{1}|^{\widetilde{p}} + a_{1}(x) |Dv_{1}|^{q}\bigg) \mathcal{I}_{\tau}^{1}(v_{1},x)=f_{1}(x) \ \ \text{in} \ \ B_{1}
\end{equation*}
in the viscosity sense, where
\begin{equation*}
    a_{1}(x):= \rho_{*}^{\widetilde{p}-q}a(\rho_{*}x), \ \mathcal{I}_{\tau}^{1}(v_{1},x):=\rho_{*}^{\tau} \mathcal{I}_{\tau}(v_{1},\rho_{*}x)
\end{equation*}
and
\begin{equation*}
  f_{1}(x):= f(\rho_{*}x)\rho_{*}^{\tau+\widetilde{p}}.
\end{equation*}
In addition, using \hyperref[A3a]{\bf (A3a)}, we obtain
\begin{equation*}
  ||f_{1}||_{L^{\infty}(B_{1})}  \leq K \rho_{*}^{\tau+\widetilde{p}+\gamma}.
\end{equation*}
Next, we select $ \delta = 10^{-(2+\alpha)} $, and let $ \epsilon $ be small positive constant satisfying    (\ref{sec4:eq1}) in Lemma \ref{se4:lemma1}. Here we choose
\begin{equation*}
  0 < \rho_{*} \leq \bigg( \frac{\epsilon}{K}  \bigg)^{\frac{1}{\tau+\widetilde{p}+\gamma}}.
\end{equation*}
Noticing that $ |v_{1}(x)|= |u(\rho_{*}x)| \leq 1+ |\rho_{*}x|^{1+\alpha_{0}} \leq 1+ |x|^{1+\alpha_{0}}$ for any $ x \in \mathbb{R}^{n} $. Then using Lemma \ref{se4:lemma1}, we get
\begin{equation*}
  \sup_{B_{1/10}(0)} v_{1}  \leq   10^{-(2+ \alpha)}.
\end{equation*}
Now defining
$$ v_{2}(x) = 10^{2+ \alpha} v_{1}\bigg(\frac{x}{10}\bigg), $$
then $ v_{2} $ solves
\begin{equation*}
  - \bigg(|Dv_{2}|^{\widetilde{p}} + a_{2}(x) |Dv_{2}|^{q}\bigg) \mathcal{I}_{\tau}^{2}(v_{2},x)=f_{2}(x) \ \ \text{in} \ \ B_{1}
\end{equation*}
in the viscosity sense, where
\begin{equation*}
  a_{2}(x) := 10^{(1+\alpha)(\widetilde{p}-q)}a_{1}(x/10), \ \mathcal{I}_{\tau}^{2}(v_{2},x):=10^{2+\alpha-\tau} \mathcal{I}_{\tau}^{1}(10^{-(2+\alpha)}v_{2}, x/10)
\end{equation*}
and
$$ f_{2}(x) = 10^{2+\alpha+\widetilde{p}(1+\alpha)-\tau} \rho_{*}^{\tau+\widetilde{p}} f\bigg(\frac{\rho_{*}x}{10}\bigg)  .    $$
Using \hyperref[A3a]{\bf (A3a)} again, we have
\begin{equation*}
  ||f_{2}||_{L^{\infty}(B_{1})} \leq K \rho_{*}^{\tau+\widetilde{p}+\gamma} 10^{2+\alpha+\widetilde{p}(1+\alpha)-\tau-\gamma}.
\end{equation*}
Notice that the choice of $ \alpha $ allows
\begin{equation*}
  2+\alpha+\widetilde{p}(1+\alpha)-\tau - \gamma =0.
\end{equation*}
It can easily be seen that $ |v_{2}(x)| \leq 1+ |x|^{1+\alpha_{0}}   $ for any $ x \in \mathbb{R}^{n} $. Then using Lemma \ref{se4:lemma1} again, we have
\begin{equation*}
  \sup_{B_{1/10}(0)} v_{2}  \leq   10^{-(2+ \alpha)}.
\end{equation*}
Scaling back to $ v_{1} $ yields
\begin{equation*}
  \sup_{B_{1/10^{2}}(0)} v_{1}  \leq   10^{-2(2+ \alpha)}.
\end{equation*}
Iterating inductively the above argument gives the following geometric decay estimate
\begin{equation*}
  \sup_{B_{1/10^{k}}(0)} v_{1}  \leq   10^{-k(2+ \alpha)}.
\end{equation*}
Finally, given $ 0 <r < \frac{\rho_{*}}{10} $, then there exists a $ k \in \mathbb{N} $ such that $ 10^{-(k+1)} < \frac{r}{\rho_{*}} \leq 10^{-k} $, therefore,
\begin{align*}
\sup_{B_{r}} u(x) & \leq \sup_{B_{r/\rho_{*}}}v_{1}(x)\leq \sup_{B_{10^{-k}}}v_{1}(x)  \\
& \leq   10^{-k(2+ \alpha)} \leq \bigg(\frac{10}{\rho_{*}}  \bigg)^{2+\alpha} r^{2+\alpha} \\
& \leq C r^{2+\alpha},
\end{align*}
where $ C= C(n,\gamma,\tau,\widetilde{p},\lambda,\Lambda) $. Thus, $ u $ is $ C^{2,\alpha} $ differentiable at $ 0 $, with $ Du(0) = D^{2}u(0) = 0 $. This completes the proof of Theorem \ref{thm2}.
\end{proof}

\vspace{3mm}

\section{Interior $ C^{1,\beta_{0}} $ regularity in the singular case}\label{Section 5}
In this section, we aim to prove $ C^{1,\beta_{0}}_{loc} $ regularity of solutions for a class non-local equation with general singularity. The argument in \cite{PDM24} will be adapted.

For every $ j \in \mathbb{N} $, we consider the following non-local uniformly elliptic equation
 \begin{equation}\label{se5:eq1}
 - \bigg( \sigma_{1}(|Du_{j}|+c_{j})  + a(x) \sigma_{2}(|Du_{j}|+c_{j})    \bigg) \mathcal{I}_{\tau}(u_{j},x) = f(x) \ \ \text{in} \ \ B_{1},
 \end{equation}
where $ (c_{j})_{j\in \mathbb{N}} \in \mathbb{R}^{+}  $ such that $ c_{j} \rightarrow 0 $ and $ c_{j} \leq 1  $ for every $ j $. First, we want to show that the sequence $(u_{j})_{j\in \mathbb{N}} $ of viscosity solution to (\ref{se5:eq1}), with boundary data in $ L_{\tau}^{1}(\mathbb{R}^{n}) $, converges to $ u_{\infty} \in C(\overline{B_{1}}) \cap L_{\tau}^{1}(\mathbb{R}^{n})    $, which is an approximated viscosity solution to equation (\ref{Pre:eq1}). To deal with this problem, we require the compactness of solution $(u_{j})_{j\in \mathbb{N}} $ for (\ref{se5:eq1}).

\begin{Proposition}
\label{se5:pro1}
Suppose $ u_{j} \in C^{0}(\overline{B}_{1})  $ is a viscosity solution to (\ref{se5:eq1}), assume \hyperref[A1]{\bf (A1)}--\hyperref[A4]{\bf (A4)} and \hyperref[A6b]{\bf (A6b)} hold. Then $ u_{j} \in C^{0,1}_{loc}(B_{1}) $ with the estimate
\begin{equation}
|u_{j}(x)-u_{j}(y)| \leq C|x-y|, \ \ \forall \ x,y \in B_{1/2},
\end{equation}
where the constant $ C > 0 $ is independent of $ j $.
\end{Proposition}
\begin{proof}
The proof of this proposition is analogous to the proof of \cite[Lemma 4.1]{PDM24}. In the sequel, we omit most of the details, focusing on the main differences with respect to argument in \cite{PDM24}. Now we split the proof into two steps.

{\bf Step 1}. Let $ \psi_{1} : \mathbb{R}^{n} \rightarrow \mathbb{R} $ be a nonnegative and smooth function such that
\begin{equation*}
\psi_{1} =0 \ \ \text{in} \  B_{1/2}, \ \ \psi_{1} = 1 \ \ \text{in} \ B_{3/4}^{c},
\end{equation*}
and define
\begin{equation*}
\psi = (\mathop{osc}\limits_{B_{1}}u_{j}+1) \psi_{1}.
\end{equation*}
In addition, let $ \varphi : [0, +\infty ) \rightarrow \mathbb{R} $ be defined as
\begin{equation*}
\varphi(t)=\left\{
     \begin{aligned}
     &   t-\frac{1}{4}t^{1+\alpha} , 0\leq t \leq t_{0}  ,        \\
     &   \varphi(t_{0}) , \ \    t > t_{0}          ,                  \\
     \end{aligned}
     \right.
\end{equation*}
where $ \alpha \in (0,1) $ are to be determined and $ t_{0}= (\frac{4}{1+\alpha})^{1/\alpha} $ is a fixed number.

Consider $ \phi, \Phi : \mathbb{R}^{n} \times \mathbb{R}^{n} \rightarrow \mathbb{R}  $ defined by
\begin{equation*}
\phi(x,y) = L \varphi(|x-y|) + \psi(y),
\end{equation*}
and
\begin{equation*}
\Phi(x,y) = u_{j}(x) - u_{j}(y) - \phi(x,y),
\end{equation*}
respectively, where $L$ is a positive constant to be determined.
Since $ \Phi $ is a continuous function, we find that $ \Phi  $ attains its maximum at $ (\overline{x}, \overline{y})  $ in $ \overline{B_{1}} \times \overline{B_{1}} $. Our goal is to show $ \Phi(\overline{x},\overline{y}) \leq 0 $. Suppose by contradiction $ \Phi(\overline{x},\overline{y}) > 0 $, then arguing as in \cite[Lemma 4.1]{PDM24}, we shall obtain two viscosity inequality as follows :
\begin{equation}\label{sec5:eq53}
\left\{
     \begin{aligned}
     &   - \bigg(\sigma_{1}(|\xi_{\overline{x}}|+c_{j}) + a(\overline{x})\sigma_{2}(|\xi_{\overline{x}}|+c_{j})               \bigg) \mathcal{I}_{\tau}(u_{j},\Phi_{\overline{y}},\overline{x}) \leq f(\overline{x}) ,        \\
     &   - \bigg(\sigma_{1}(|\xi_{\overline{y}}|+c_{j}) + a(\overline{y})\sigma_{2}(|\xi_{\overline{y}}|+c_{j})                        \bigg) \mathcal{I}_{\tau}(u_{j},-\Phi_{\overline{x}},\overline{y}) \geq f(\overline{y}) ,                  \\
     \end{aligned}
     \right.
\end{equation}
where
\begin{equation*}
\left\{
     \begin{aligned}
     &  \xi_{\overline{x}} := D_{x} \phi(\overline{x},\overline{y}), \ \ \xi_{\overline{y}} := D_{y} \phi(\overline{x},\overline{y}),      \\
     &  \Phi_{\overline{y}}(x) = u_{j}(\overline{y}) + \phi(x, \overline{y}), \ \ - \Phi_{\overline{x}}(y) = -u_{j}(y) - \phi(\overline{x},y).   \\
     \end{aligned}
     \right.
\end{equation*}

{\bf Step 2}. Notice that
\begin{equation*}
  |\xi_{\overline{x}}| \leq L(2+\alpha),
\end{equation*}
which, together with (\ref{sec5:eq53}), \hyperref[A4]{\bf (A4)} and \hyperref[A6b]{\bf (A6b)}, yields that
\begin{align*}
- \mathcal{I}_{\tau}(u_{j},\Phi_{\overline{y}},\overline{x}) & \leq \frac{f(\overline{x})}{\sigma_{1}(|\xi_{\overline{x}}|+c_{j}) + a(\overline{x})\sigma_{2}(|\xi_{\overline{x}}|+c_{j})}  \\
& \leq \frac{||f||_{L^{\infty}(B_{1})}}{\sigma_{2}(1+L(2+\alpha))}.
\end{align*}
Similarly, we have
\begin{equation*}
  -\mathcal{I}_{\tau}(u_{j},-\Phi_{\overline{x}},\overline{y}) \geq -\frac{||f||_{L^{\infty}(B_{1})}}{\sigma_{2}(1+L(2+\alpha))},
\end{equation*}
and consequently
\begin{equation*}
\mathcal{I}_{\tau}(u_{j},\Phi_{\overline{y}},\overline{x}) -\mathcal{I}_{\tau}(u_{j},-\Phi_{\overline{x}},\overline{y})  \geq -\frac{2||f||_{L^{\infty}(B_{1})}}{\sigma_{2}(1+L(2+\alpha))}.
\end{equation*}
At this stage, our aim is to estimate the left-hand side of inequality above. However, this argument is identical to the proof of \cite[Lemma 4.1]{PDM24}, so we will not repeat the process and give the final estimate directly
\begin{equation}\label{sec5:eq54}
- C\big( ||f||_{L^{\infty}(B_{1})}  + \mathop{osc}_{B_{1}} u + ||u||_{L^{1}_{\tau}(\mathbb{R}^{n})}   +  1  \big)  \leq  - c|\overline{x} - \overline{y}|^{1-\tau+\alpha(n+2-\tau)}.
\end{equation}
Now we choose $ 0< \alpha < 1 $ small enough such that
\begin{equation}\label{sec5:eq55}
- \widetilde{\beta} := 1- \tau + \alpha (n+2-\tau) \leq \frac{1-\tau}{2} < 0.
\end{equation}
It is easy to check that $ |\overline{x} - \overline{y}| \leq 2L^{-1} \mathop{osc}_{B_{1}} u $, which together with \eqref{sec5:eq54} and \eqref{sec5:eq55} to arrive
 \begin{equation*}
   - C_{1}\big( ||f||_{L^{\infty}(B{1})}  + \mathop{osc}_{B_{1}} u + ||u||_{L^{1}_{\tau}(\mathbb{R}^{n})}   +  1  \big)  \leq  -L,
 \end{equation*}
which is a contradiction, provided we choose $ L $ large enough. Hence, we obtain $ \Phi(\overline{x},\overline{y}) \leq 0 $. This finish the proof of proposition.
\end{proof}

Analogous to \cite[Proposition 4.1]{PDM24} we can also show the existence of an approximated viscosity solution $ u_{\infty} \in C(\overline{B_{1}}) \cap L_{\tau}^{1}(\mathbb{R}^{n})$ for equation (\ref{Pre:eq1}). For the sake of completeness, we present it here.

\begin{Proposition}
Assume \hyperref[A1]{\bf (A1)}--\hyperref[A4]{\bf (A4)} and \hyperref[A6b]{\bf (A6b)} hold. Then there exists at least an approximated viscosity solution $ u \in C^{0}(B_{1}) $ to (\ref{Pre:eq1}).
\end{Proposition}
\begin{proof}
Consider the following Dirichlet problem
\begin{equation*}
\left\{
     \begin{aligned}
     &   - \bigg( \sigma_{1}(|Du_{j}|+\frac{1}{j})  + a(x) \sigma_{2}(|Du_{j}|+ \frac{1}{j})    \bigg) \mathcal{I}_{\tau}(u_{j},x) = f(x) \ \ \text{in} \ \ B_{1}       \\
     &    u_{j}(x) = g(x) \ \ \ \ \ \ \ \ \ \ \ \ \ \ \text{in}  \ \ \mathbb{R}^{n} \setminus B_{1}    ,                  \\
     \end{aligned}
     \right.
\end{equation*}
where the boundary datum satisfies
\begin{equation*}
  | g(x) |  \leq 1+ |x|^{1+\alpha} \ \ \text{for} \ \ x \in \mathbb{R}^{n}
\end{equation*}
with any $ 1+ \alpha \in (0, \tau) $. The existence of viscosity solution $ u_{j} $ for equation above follows from \cite{BCI08}, since the non-local operator $ \mathcal{I}_{\tau} $ here is uniformly elliptic. In the spirit of Proposition \ref{se5:pro1}, we have $ u_{j} \in C^{0,1}_{loc}(B_{1}) $. Thereby, there exists $ u_{\infty} \in C^{0,\widetilde{\alpha}}_{loc}(B_{1})$ for some $ \widetilde{\alpha} \in (0,1) $ such that $ u_{j}   $ converges locally uniformly to $ u_{\infty} $ in $ B_{1} $ and $ u_{\infty} = g $ in $  \mathbb{R}^{n} \setminus B_{1} $. By taking $ c_{j} : = \frac{1}{j} $, then from Definition \ref{def22} in Section \ref{Section 2}, we know $ u_{\infty} $ is an approximated viscosity solution to (\ref{Pre:eq1}).
\end{proof}

Resorting to Proposition \ref{se5:pro1}, we will show that interior $ C^{1,\beta_{0}} $ regularity of viscosity solution $ (u_{j})_{j\in \mathbb{N}} $ to equation (\ref{se5:eq1}).

\begin{Corollary}\label{se5:coro1}
Under the assumptions of Proposition \ref{se5:pro1}, then $ u_{j} \in C^{1,\beta_{0}}_{loc}(B_{1}) $ with the estimate
  \begin{equation*}
    ||u_{j}||_{C^{1,\beta_{0}}(B_{1/2})} \leq C(||u_{j}||_{L^{\infty}(B_{1})}+||f||_{L^{\infty}(B_{1})}),
  \end{equation*}
  where $ C = C (n,c_{0},\lambda,\Lambda,\tau) $ is a positive constant.
\end{Corollary}
\begin{proof}
We note from \hyperref[A6b]{\bf (A6b)} that $ \sigma_{i}(t) \geq \frac{c_{0}}{2} $ for $ 0 < t <t_{0} $. By Proposition \ref{se5:pro1} and $ c_{j} \leq 1 $, it allows
\begin{equation*}
  \sigma_{2}(|Du_{j}|+c_{j}) \geq \frac{c_{0}}{2}.
\end{equation*}
From \hyperref[A4]{\bf (A4)} and (\ref{se5:eq1}), we obtain
\begin{equation*}
  - \mathcal{I}_{\tau}(u_{j},x) \leq \frac{2}{c_{0}} ||f||_{L^{\infty}(B_{1})}, \ \ \text{and} \ \ - \mathcal{I}_{\tau}(u_{j},x) \geq -\frac{2}{c_{0}} ||f||_{L^{\infty}(B_{1})}.
\end{equation*}
Hence by \cite[Theorem 52]{LL11}, we have $ u_{j} \in C^{1,\beta_{0}}_{loc}(B_{1}) $ with the estimate
 \begin{equation*}
    ||u_{j}||_{C^{1,\beta_{0}}(B_{1/2})} \leq C(||u_{j}||_{L^{\infty}(B_{1})}+||f||_{L^{\infty}(B_{1})}).
  \end{equation*}
 This completes the proof of the Corollary \ref{se5:coro1}.
\end{proof}

With the help of Corollary \ref{se5:coro1}, we shall prove the interior $ C^{1,\beta_{0}} $ regularity of solution to (\ref{Pre:eq1}).

\begin{proof}[Proof of Theorem~\ref{thm3}]
From Definition \ref{def22}, there are sequences $(u_{j})_{j\in\mathbb{N}} \in C(\overline{B_{1}})\cap L_{\tau}^{1}(\mathbb{R}^{n})  $, $ (c_{j})_{j\in\mathbb{N}} \in \mathbb{R}^{+}  $, $ c_{j} \rightarrow 0 $ and $ u_{j} \rightarrow u $ locally uniformly in $ B_{1} $, such that $ u_{j} $ is a viscosity solution to
\begin{equation*}
 - \bigg( \sigma_{1}(|Du_{j}|+c_{j})  + a(x) \sigma_{2}(|Du_{j}|+c_{j})    \bigg) \mathcal{I}_{\tau}(u_{j},x) = f(x) \ \ \text{in} \ \ B_{1}.
 \end{equation*}
From Corollary \ref{se5:coro1}, it follows that $ u_{j} \in C^{1,\beta_{0}}_{loc}(B_{1}) $ and
 \begin{equation}\label{se5:eq3}
    ||u_{j}||_{C^{1,\beta_{0}}(B_{1/2})} \leq C(||u_{j}||_{L^{\infty}(B_{1})}+||f||_{L^{\infty}(B_{1})}).
  \end{equation}
Taking the limit $ j \rightarrow \infty $ in (\ref{se5:eq3}) and combining the fact $ u_{j} \rightarrow u $ locally uniformly in $ B_{1} $, there holds
$ u \in C^{1,\beta_{0}}_{loc}(B_{1}) $ and
\begin{equation*}
    ||u||_{C^{1,\beta_{0}}(B_{1/2})} \leq C(||u||_{L^{\infty}(B_{1})}+||f||_{L^{\infty}(B_{1})}).
\end{equation*}
The proof of Theorem~\ref{thm3} is completed.
\end{proof}

\vspace{3mm}

\section{Borderline regularity of a non-local $ p$-Laplacian equation}\label{section 6}

In this section, we turn to study a degenerate non-local normalized $ p$-Laplacian equation. Borderline regularity of the solutions are proved similarly to Section \ref{Section 3}. Therefore, the result of this section can be regarded as a byproduct of the result in Section \ref{Section 3}.

Before proceeding, we will give some notations and definitions which will be used in this section.
Recall that $ \Delta_{p,N}^{s}u $ denotes the non-local normalized $ p$-Laplacian operator given by
\begin{equation}\label{Sec6:eq111}
\Delta_{p,N}^{s}u := (1-s)
\sup_{\xi_{1}\in \mathbf{A}}\inf_{\xi_{2}\in \mathbf{A}} \bigg(\int_{\mathbb{R}^{n}}\frac{[u(x+y)-u(x)]K_{\xi_{1}}(y)+[u(x-y)-u(x)]K_{\xi_{2}}(y)}{|y|^{n+2s}}dy\bigg),
\end{equation}
with
\begin{equation*}
 s \in (0,1), \ \ \mathbf{A}:=S^{n-1}, \ \  K_{\xi}(y):=\left\{
     \begin{aligned}
     &   \frac{1}{\alpha_{p}} \mathbb{I}_{[c_{p},1]}\bigg( \frac{y}{|y|}\cdot \xi\bigg), \ \ p \geq 2 ,   \\
     &   \frac{1}{\alpha_{p}} \mathbb{I}_{[0,c_{p}]}\bigg( \frac{y}{|y|}\cdot \xi\bigg), \ \  1<p<2           .                 \\
     \end{aligned}
     \right.
\end{equation*}
We only restrict ourselves to the case $ s \in (1/2,1) $ and $ p \in [2,\infty) $. While all of results hold for the case $ p \in (1,2) $, see, for example, \cite[Section 4]{CLA12}.

To describe the definition of viscosity solutions to
\begin{equation}\label{App:eq1}
  - \bigg( \sigma_{1}(|Du|)  + a(x) \sigma_{2}(|Du|) \bigg)  \Delta_{p,N}^{s}u = f(x) \ \ \text{in} \ \ B_{1},
\end{equation}
we first define $ \Delta_{p,N,+}^{s} $ and $ \Delta_{p,N,-}^{s} $ in the following way :

$ \bullet $ If $ \nabla u(x) \neq 0 $, then
\begin{align*}
\Delta_{p,N,+}^{s} & = \Delta_{p,N,-}^{s}     \\
&  : = \frac{1-s}{\alpha_{p}} \int_{\mathbb{R}^{n}} \frac{[u(x+y)+u(x-y)-2u(x)]\mathbb{I}_{[c_{p},1]}\big(\frac{y}{|y|}\cdot \xi     \big)}{|y|^{n+2s}} dy, \ \ \xi = \frac{\nabla u(x)}{|\nabla u(x)|}.
\end{align*}

$ \bullet $ If $ \nabla u(x) = 0 $, then
\begin{align*}
& \Delta_{p,N,+}^{s}:= \frac{1-s}{\alpha_{p}} \sup_{\xi \in S^{n-1}}\int_{\mathbb{R}^{n}} \frac{[u(x+y)+u(x-y)-2u(x)]\mathbb{I}_{[c_{p},1]}\big(\frac{y}{|y|}\cdot \xi     \big)}{|y|^{n+2s}} dy, \\
& \Delta_{p,N,-}^{s}:= \frac{1-s}{\alpha_{p}} \inf_{\xi \in S^{n-1}}\int_{\mathbb{R}^{n}} \frac{[u(x+y)+u(x-y)-2u(x)]\mathbb{I}_{[c_{p},1]}\big(\frac{y}{|y|}\cdot \xi     \big)}{|y|^{n+2s}} dy,
\end{align*}
where
\begin{equation*}
\alpha_{p} := \frac{1}{2} \int_{\partial B_{1}} (\omega \cdot e_{2})^{2} \mathbb{I}_{[c_{p},1]}(\omega \cdot e_{1})d\sigma(\omega),
\end{equation*}
and
\begin{equation*}
\beta_{p} := \frac{1}{2} \int_{\partial B_{1}} (\omega \cdot e_{1})^{2} \mathbb{I}_{[c_{p},1]}(\omega \cdot e_{1})d\sigma(\omega) - \alpha_{p}.
\end{equation*}
For $ p \geq 2 $, we choose $ c_{p} \in [0,1] $ such that $ \beta_{p} / \alpha_{p} = p-2   $. Such a choice is possible for any $ p \in [2,\infty) $.

Next we present the definition of viscosity solution to (\ref{App:eq1}).
\begin{Definition}\label{App:De1}
An upper (resp. lower) semi-continuous function $ u: \mathbb{R}^{n} \rightarrow \mathbb{R}  $ is said to be a viscosity subsolution (resp. supersolution)  to (\ref{App:eq1}) at $ x_{0} \in B_{1} $, if every time a test function $ \varphi \in C^{1,1}(x_{0}) $ touches $ u $ from above (resp. below) at $ x_{0} $, we have
\begin{align*}
  &- \bigg( \sigma_{1}(|D\varphi(x_{0})|)  + a(x_{0}) \sigma_{2}(|D\varphi(x_{0})|) \bigg)  \Delta_{p,N}^{s} \widetilde{u}(x_{0}) \leq f(x_{0}),   \\
  & \bigg( resp. - \bigg( \sigma_{1}(|D\varphi(x_{0})|)  + a(x_{0}) \sigma_{2}(|D\varphi(x_{0})|) \bigg)  \Delta_{p,N}^{s} \widetilde{u}(x_{0}) \geq f(x_{0}) \bigg).
\end{align*}
where $ \widetilde{u} $ is defined as
\begin{equation*}
\widetilde{u}(x):=\left\{
     \begin{aligned}
     &   \varphi(x), \ \ \text{if} \ x \in B_{r}(x_{0}),     \\
     &  u(x), \ \  \text{if} \  x \in \mathbb{R}^{n} \setminus B_{r}(x_{0}).                \\
     \end{aligned}
     \right.
\end{equation*}
If a function is both a viscosity subsolution and supersolution to (\ref{App:eq1}), we say it is a viscosity solution to (\ref{App:eq1}).
\end{Definition}

We now formulate the main result of this section.
\begin{Theorem}\label{thm4}(borderline regularity in the degenerate case)
Let $ u \in C^{0}(\overline{B}_{1}) \cap L_{s}^{1}(\mathbb{R}^{n}) $ be a bounded viscosity solution to (\ref{App:eq1}), assume \hyperref[A3]{\bf (A3)}, \hyperref[A4]{\bf (A4)}, \hyperref[A6a]{\bf (A6a)}, \hyperref[A7]{\bf (A7)} hold, $ s \in (1/2,1)$ and $p \geq2$. Then there exists $ s_{0} \in (1/2,1) $ sufficiently close to $ 1 $ such that if $ s \in (s_{0},1) $, then $ u \in C^{1}_{loc}(B_{1})$. Moreover, there exists a modulus of continuity $ \omega: \mathbb{R}_{0}^{+} \rightarrow  \mathbb{R}_{0}^{+}$ depending only on $ n, s, \sigma_{1}, \sigma_{2}, ||u||_{L^{\infty}(B_{1})}$ and $||f||_{L^{\infty}(B_{1})} $ such that
\begin{equation*}
 |Du(x)-Du(y)| \leq \omega (|x-y|),
\end{equation*}
for every $ x,y \in B_{1/4} $.
\end{Theorem}

To the best of our knowledge, the relevant result has not appeared even for the simple prototype given by the form $ -|Du|^{\gamma} \Delta_{p,N}^{s} u = 1 $, with $ \sigma_{1}(|Du|)= |Du|^{\gamma}, \gamma > 0, a(x) \equiv 0 $ and $ f \equiv 1 $. Therefore, Theorem \ref{thm4} slightly generalizes the partial regularity result of \cite[Therem 1.1]{AE18}. The obstruction of this slight extension is twofold, namely,

\begin{itemize}

\item  The operator (\ref{Sec6:eq111}) is not, in general, elliptic in the sense of \cite[Definition 3.1]{LL09}, and hence the argument of Theorem \ref{thm1} cannot be applied directly to (\ref{App:eq1}).

\item  The higher regularity of solutions associated with operators of the form (\ref{Sec6:eq111}) appears to be a very delicate question due to the subtle interaction of local and non-local effects. Additionally, the problem becomes highly degenerate as soon as $ \sigma_{i}(|Du|), i=1,2 $ vanishes.

\end{itemize}

To overcome these difficulties, here we need to combine the arguments established in Theorem \ref{thm1} with the celebrated work of Bjorland-Caffarelli-Figalli \cite{CLA12}. To simplify the exposition, we assume that $ u \in C(B_{1}) \cap L_{s}^{1}(\mathbb{R}^{n}) $ is a normalized viscosity solution if
$ ||u||_{L^{\infty}(B_{1})} \leq 1 $. Similar to Remark \ref{RK:22}, we also assume
$ ||f||_{L^{\infty}(B_{1})} \leq \epsilon $ for a given $ \epsilon > 0 $.
As mentioned before, we first prove the compactness of solutions to a variant of (\ref{App:eq1}).

\begin{Lemma}\label{Sec6:lem1}
Let $ u \in C^{0}(\overline{B}_{1}) \cap L_{s}^{1}(\mathbb{R}^{n}) $ be a normalized viscosity solution to
\begin{equation*}\label{The:eq1}
- \bigg( \sigma_{1}(|Du+\xi|)  + a(x) \sigma_{2}(|Du+\xi|)    \bigg)  \Delta_{p,N}^{s}u = f(x) \ \ \text{in} \ \ B_{1},
\end{equation*}
where $ \xi \in \mathbb{R}^{n}  $ is arbitrary, assume \hyperref[A3]{\bf (A3)}, \hyperref[A4]{\bf (A4)} and \hyperref[A6a]{\bf (A6a)} hold, and $ s \in (1/2,1), p \geq 2 $. Then $ u $ is locally H\"{o}lder continuous in $ B_{1} $, i.e., for every $ x,y \in B_{1/2} $,
\begin{equation*}
 |u(x)-u(y)| \leq C |x-y|^{\beta} \ \ \text{for some} \  \beta \in (0,1),
\end{equation*}
where the constant $ C > 0 $ is independent of $ \xi $.
\begin{proof}
We divide the proof into two different cases in which $ |\xi| $ is large or small.

Case 1: $ |\xi| \geq A_{0}$. Fixing $ r \in (0,1) $ and $ x_{0} \in B_{1/2}   $, we shall show that there exist two positive constants $ L_{1} ,L_{2} $ such that
\begin{equation}\label{Sec6:extreq1}
M:=\sup_{x,y \in B_{r}(x_{0})} \bigg( u(x)-u(y)-L_{1}h(|x-y|)-L_{2} (|x-x_{0}|^{2}+|y-x_{0}|^{2}) \bigg) \leq 0   ,
\end{equation}
where
\begin{equation}\label{Sec6:eq123}
h(t)£º=\left\{
     \begin{aligned}
     &   t-\frac{1}{2}t^{3/2} , 0\leq t < t_{0}  ,        \\
     &   h(t_{0}) , \ \    t\geq t_{0}          ,                  \\
     \end{aligned}
     \right.
\end{equation}
with $ t_{0} = \frac{16}{9}>1 $ and $ A_{0} $ is to be determined.

Now we argue by contradiction. Assuming that for any $ L_{1}, L_{2} >0  $, $ M  $ attains its positive maximum at $ (\overline{x},\overline{y}) \in B_{r}(x_{0})    $, that is to say,
\begin{equation}\label{the:eq1}
L_{1}h(|\overline{x}-\overline{y}|)+L_{2} (|\overline{x}-x_{0}|^{2}+|\overline{y}-x_{0}|^{2}) \leq 2||u||_{L^{\infty}(B_{1})} .
\end{equation}
For simplicity, we denote
$$ \varphi (x,y):= L_{1}h(|x-y|)+L_{2} (|x-x_{0}|^{2}+|y-x_{0}|^{2})   ,   $$
and
$$ \phi(x,y):=  u(x)-u(y)-\varphi (x,y).$$
Notice that $ \overline{x} \neq \overline{y} $. Indeed, if $ \overline{x} = \overline{y} $, then it is easily seen that $ M \leq 0 $, which contradicts with the positivity of $ M $.

Now choosing $ L_{2} =\frac{32}{r^{2}}||u||_{L^{\infty}(B_{1})}  $, since (\ref{the:eq1}), we have
\begin{equation*}
|\overline{x}-x_{0}|+|\overline{y}-y_{0}| < \frac{1}{2}r.
\end{equation*}
This implies that $ \overline{x},\overline{y}\in B_{r}(x_{0})    $.

By Lemma \ref{lem2.3}, we can ensure the existence of limiting subjet $ (q_{\overline{x}},\text{X})$ of $ u $ at $ \overline{x} $ and limiting superjet $ (q_{\overline{y}},\text{Y})$ of $ u $ at $ \overline{y} $ and the following matrix inequality holds:
\begin{equation}\label{the:eq2}
-\frac{1}{\delta} \textbf{I} \leq
\begin{pmatrix}
\text{X}  &   0   \\
0  &    -\text{Y}
\end{pmatrix}
\leq
\begin{pmatrix}
\text{Z}   &   -\text{Z}   \\
-\text{Z}  &    \text{Z}
\end{pmatrix}
 + (2L_{2}+\delta)\textbf{I} , \ \  0 < \delta \ll 1 ,
\end{equation}
where
\begin{equation*}
\label{eq4.10}
\text{Z}=L_{1}h''(|\overline{x}-\overline{y}|)\frac{\overline{x}-\overline{y}}{|\overline{x}-\overline{y}|}\otimes \frac{\overline{x}-\overline{y}}{|\overline{x}-\overline{y}|} +\frac{L_{1}h'(|\overline{x}-\overline{y}|)}{|\overline{x}-\overline{y}|}\left( \textbf{I}- \frac{\overline{x}-\overline{y}}{|\overline{x}-\overline{y}|}\otimes \frac{\overline{x}-\overline{y}}{|\overline{x}-\overline{y}|} \right),
\end{equation*}
and
\begin{equation}\label{the:eq3}
\left\{
     \begin{aligned}
     & q_{\overline{x}}:=\partial_{\overline{x}}\varphi(\overline{x},\overline{y}) =L_{1}h'(|\overline{x}-\overline{y}|)\frac{\overline{x}-\overline{y}}{|\overline{x}-\overline{y}|}+2L_{2}(\overline{x}-x_{0}) ,      \\
     & q_{\overline{y}}:=-\partial_{\overline{y}}\varphi(\overline{x},\overline{y}) =L_{1}h'(|\overline{x}-\overline{y}|)\frac{\overline{x}-\overline{y}}{|\overline{x}-\overline{y}|}-2L_{2}(\overline{y}-x_{0}) .           \\
     \end{aligned}
     \right.
\end{equation}

Then we obtain the following two viscosity inequalities :
\begin{equation}\label{the:eq6}
- \bigg( \sigma_{1}(|q_{\overline{x}}+\xi|) + a(\overline{x}) \sigma_{2}(|q_{\overline{x}}+\xi|)    \bigg) \Delta_{p,N}^{s} \widetilde{u}(\overline{x}) \leq ||f||_{L^{\infty}(B_{1})},
\end{equation}
and
\begin{equation}\label{the:eq7}
-\bigg( \sigma_{1}(|q_{\overline{y}}+\xi|) + a(\overline{y}) \sigma_{2}(|q_{\overline{y}}+\xi|)    \bigg)  \Delta_{p,N}^{s} \widetilde{u}(\overline{y}) \geq -||f||_{L^{\infty}(B_{1})}.
\end{equation}
Subtracting (\ref{the:eq6}) from (\ref{the:eq7}) yields
\begin{align}\label{the:eq8}\begin{split}
 &\underbrace{ \bigg ( \Delta_{p,N}^{s} \widetilde{u}(\overline{x})-\Delta_{p,N}^{s} \widetilde{u}(\overline{y})\bigg)}_{:=\text{B}} \\
 &\geq \underbrace{-||f||_{L^{\infty}(B_{1})} \bigg(\frac{1}{\sigma_{1}(|q_{\overline{x}}+\xi|) + a(\overline{x}) \sigma_{2}(|q_{\overline{x}}+\xi|)} + \frac{1}{\sigma_{1}(|q_{\overline{y}}+\xi|) + a(\overline{y}) \sigma_{2}(|q_{\overline{y}}+\xi|)}\bigg)}_{:=\text{D}}. \end{split}
\end{align}

(1) The estimate of $ \text{B}$. We split the term $  \text{B} = \text{B.1} + \text{B.2} + \text{B.3} + \text{B.4}    $, where
\begin{equation*}
\begin{aligned}
&& \text{B.1} & =\frac{1-s}{\alpha_{p}} \bigg(\int_{B_{r_{0}}}\frac{\big[\varphi(\overline{x}+z,\overline{y})+\varphi(\overline{x}-z,\overline{y})-2\varphi(\overline{x},\overline{y})\big]\mathbb{I}_{[c_{p},1]}\big(\frac{z}{|z|}\cdot\frac{q_{\overline{x}}+\xi}{|q_{\overline{x}}+\xi|} \big)}{|z|^{n+2s}}dz\bigg), \\
&& \text{B.2} & =-\frac{1-s}{\alpha_{p}} \bigg(\int_{B_{r_{0}}}\frac{\big[\varphi(\overline{x},\overline{y}+z)+\varphi(\overline{x},\overline{y}-z)-2\varphi(\overline{x},\overline{y})\big]\mathbb{I}_{[c_{p},1]}\big(\frac{z}{|z|}\cdot\frac{q_{\overline{y}}+\xi}{|q_{\overline{y}}+\xi|} \big)}{|z|^{n+2s}}dz\bigg), \\
&& \text{B.3} & = \frac{1-s}{\alpha_{p}} \bigg(\int_{\mathbb{R}^{n}\setminus B_{r_{0}}}\frac{\big[u(\overline{x}+z)+u(\overline{x}-z)-2u(\overline{x})\big]\mathbb{I}_{[c_{p},1]}\big(\frac{z}{|z|}\cdot\frac{Du(\overline{x})+\xi}{|Du(\overline{x})+\xi|} \big)}{|z|^{n+2s}}dz\bigg), \\
&& \text{and}     \\
&& \text{B.4} & = - \frac{1-s}{\alpha_{p}} \bigg(\int_{\mathbb{R}^{n}\setminus B_{r_{0}}}\frac{\big[u(\overline{y}+z)+u(\overline{y}-z)-2u(\overline{y})\big]\mathbb{I}_{[c_{p},1]}\big(\frac{z}{|z|}\cdot\frac{Du(\overline{y})+\xi}{|Du(\overline{y})+\xi|} \big)}{|z|^{n+2s}}dz\bigg).
\end{aligned}
\end{equation*}

Estimate for the term $ \text{B.1} $. Due to the formula of $ \varphi (x,y) $, it yields that
\begin{align}\label{the:eq10}
\text{B.1} &=  \frac{L_{1}(1-s)}{\alpha_{p}} \int_{B_{r_{0}}}\frac{\big[ h(|\overline{x}+z-\overline{y}|) +h(|\overline{x}-z-\overline{y}|)-2h(|\overline{x}-\overline{y}|)\big]\mathbb{I}_{[c_{p},1]}\big(\frac{z}{|z|}\cdot\frac{q_{\overline{x}}+\xi}{|q_{\overline{x}}+\xi|} \big)}{|z|^{n+2s}}dz     \\
& \ \ \ \ + \frac{2L_{2}(1-s)}{\alpha_{p}} \int_{B_{r_{0}}} \frac{1}{|z|^{n+2s-2}}\mathbb{I}_{[c_{p},1]}\bigg(\frac{z}{|z|}\cdot\frac{q_{\overline{x}}+\xi}{|q_{\overline{x}}+\xi|} \bigg)dz,  \nonumber  \\
&  = : \text{F.1} + \text{F.2}  \nonumber  .
\end{align}

For the term $ \text{F.1}   $, using a second-order Taylor expansion and integrating by parts, it follows that
\begin{align}\label{the:eq11}
\begin{split}
\text{F.1}&=  \frac{L_{1}(1-s)}{\alpha_{p}} D^{2}h(|\overline{x}-\overline{y}|) : \int_{B_{r_{0}}} \frac{z\otimes z}{|z|^{n+2s}}\mathbb{I}_{[c_{p},1]}\bigg(\frac{z}{|z|}\cdot\frac{q_{\overline{x}}+\xi}{|q_{\overline{x}}+\xi|} \bigg)dz     \\
& \ \ \ + \frac{L_{1}(1-s)}{\alpha_{p}} \int_{B_{r_{0}}} \frac{o(|z|^{2})}{|z|^{n+2s}}\mathbb{I}_{[c_{p},1]}\bigg(\frac{z}{|z|}\cdot\frac{q_{\overline{x}}+\xi}{|q_{\overline{x}}+\xi|} \bigg)dz \\
&= r_{0}^{2-2s} \frac{L_{1}(1-s)}{2\alpha_{p}} D^{2}h(|\overline{x}-\overline{y}|) : \int_{\partial B_{1}} \omega \otimes \omega\mathbb{I}_{[c_{p},1]}\bigg(\omega\cdot\frac{q_{\overline{x}}+\xi}{|q_{\overline{x}}+\xi|} \bigg)dz   \\
& \ \ \ + \frac{L_{1}(1-s)}{\alpha_{p}} \int_{B_{r_{0}}}   \frac{o(|z|^{2})}{|z|^{n+2s}}\mathbb{I}_{[c_{p},1]}\bigg(\frac{z}{|z|}\cdot\frac{q_{\overline{x}}+\xi}{|q_{\overline{x}}+\xi|} \bigg)dz  \\
& = r_{0}^{2-2s} L_{1} \text{Tr} \bigg( \big( \textbf{I}+ \frac{\beta_{p}}{\alpha_{p}} \frac{q_{\overline{x}}+\xi}{|q_{\overline{x}}+\xi|}\otimes \frac{q_{\overline{x}}+\xi}{|q_{\overline{x}}+\xi|}     \big) D^{2}h(|\overline{x}-\overline{y}|) \bigg)  \\
& \ \ \ + \frac{L_{1}(1-s)}{\alpha_{p}} \int_{B_{r_{0}}}   \frac{o(|z|^{2})}{|z|^{n+2s}}\mathbb{I}_{[c_{p},1]}\bigg(\frac{z}{|z|}\cdot\frac{q_{\overline{x}}+\xi}{|q_{\overline{x}}+\xi|} \bigg)dz \\
& = : \text{F.1.1} + \text{F.1.2} \ .
\end{split}
\end{align}
For the term $ \text{F.1.1} $, direct computations yield that
\begin{align*}
\bigg \langle D^{2}h(|\overline{x}-\overline{y}|)\widehat{\delta}, \widehat{\delta} \bigg \rangle & = \bigg \langle \bigg( h''(|\overline{x}-\overline{y}|)\widehat{\delta} \otimes \widehat{\delta} +\frac{h'(|\overline{x}-\overline{y}|)}{|\overline{x}-\overline{y}|}( \textbf{I}- \widehat{\delta}\otimes \widehat{\delta}) \bigg)\widehat{\delta}, \widehat{\delta} \bigg \rangle   \\
& = -\widehat{\delta}^{T}\big( \widehat{\delta} \otimes \widehat{\delta}\big)\widehat{\delta} +\frac{1-|\overline{x}-\overline{y}|}{|\overline{x}-\overline{y}|} \widehat{\delta}^{T} \big( \textbf{I}- \widehat{\delta} \otimes \widehat{\delta} \big)\widehat{\delta}=-1 <0,
\end{align*}
where $ \widehat{\delta} =\frac{\overline{x}-\overline{y}}{|\overline{x}-\overline{y}|}    $. This implies that at least one eigenvalue of $ D^{2}h(|\overline{x}-\overline{y}|) $ is equal $ -1 $. Thus it reads
\begin{align}\label{the:eq12}
\begin{split}
\text{F.1.1} & = r_{0}^{2-2s} L_{1} \sum_{i=1}^{n} \lambda_{i} \bigg(\textbf{I}+ \frac{\beta_{p}}{\alpha_{p}} \frac{q_{\overline{x}}+\xi}{|q_{\overline{x}}+\xi|}\otimes \frac{q_{\overline{x}}+\xi}{|q_{\overline{x}}+\xi|}\bigg) \lambda_{i} (D^{2}h(|\overline{x}-\overline{y}|))    \\
& \leq -r_{0}^{2-2s} \min \bigg\{1,p-1\bigg\} L_{1}.
\end{split}
\end{align}
For the term $ \text{F.1.2} $, it is easy to check
\begin{equation}\label{the:eq13}
\text{F.1.2} \leq C L_{1} o(1),
\end{equation}
for some positive constant $ C $.

For the term $ \text{F.2} $, simple calculations yield that
\begin{equation}\label{the:eq14}
\text{F.2} \leq \frac{CL_{2}}{\alpha_{p}}r_{0}^{2-2s}, \ \text{for some constant} \ C>0.
\end{equation}
Hence, we combine the above estimate (\ref{the:eq10})--(\ref{the:eq13}) and (\ref{the:eq14}) to obtain
\begin{align}\label{the:eq15}
\text{B.1} \leq -r_{0}^{2-2s} \min \bigg\{1,p-1\bigg\} L_{1} + C L_{1} o(1) + \frac{CL_{2}}{\alpha_{p}}r_{0}^{2-2s}.
\end{align}

Estimate for the term $ \text{B.3} $. Since
\begin{equation*}
\frac{1}{|z|^{n+2s}} \leq \frac{2 r_{0}^{-n-2s}}{1+|z|^{n+2s}}
\end{equation*}
for $ |z|\geq r_{0} $, and hence we obtain the following estimate for the term $ \text{B.3} $,
\begin{align}\label{the:eq16}
\begin{split}
\text{B.3} & \leq 4 r_{0}^{-n-2s} \frac{1-s}{\alpha_{p}} \int_{\mathbb{R}^{n}} \frac{|u(\overline{x}+z)-u(\overline{x})|}{1+|z|^{n+2s}} \\
& \leq C r_{0}^{-n-2s} \frac{1-s}{\alpha_{p}} ||u||_{L^{1}_{s}(\mathbb{R}^{n})} .
\end{split}
\end{align}

Noticing the estimate of $ \text{B.2} $ and $ \text{B.4} $ are similar to that of $ \text{B.1} $ and $ \text{B.3} $, respectively. Thus, we collect (\ref{the:eq15})--(\ref{the:eq16}) and the formula of $ \text{B} $ to obtain
\begin{align}\label{the:eq17}
\begin{split}
\text{B} \leq 2 \beta \bigg( -r_{0}^{2-2s} \min \bigg\{1,p-1\bigg\} L_{1} & + C L_{1} o(1) + \frac{CL_{2}}{\alpha_{p}}r_{0}^{2-2s} \bigg)   \\
&  +  2\beta C r_{0}^{-n-2s} \frac{1-s}{\alpha_{p}} ||u||_{L^{1}_{s}(\mathbb{R}^{n})} .
\end{split}
\end{align}

(2) The estimate of $\text{D}$. From (\ref{the:eq3}), we see that
\begin{equation*}
|q_{\overline{x}}|,|q_{\overline{y}}|\leq L_{1}  +2L_{2}.
\end{equation*}
Choosing $ A_{0} =3L_{1} +2L_{2}$, and using \hyperref[A6a]{\bf (A6a)}, it follows that
\begin{equation}\label{the:eq112}
\left\{
     \begin{aligned}
     &  \sigma_{1}(|q_{\overline{x}}+\xi|) + a(\overline{x}) \sigma_{2}(|q_{\overline{x}}+\xi|) \geq \sigma_{1}(|q_{\overline{x}}+\xi|)\geq \sigma_{1}(2L_{1})    ,      \\
     &  \sigma_{1}(|q_{\overline{y}}+\xi|) + a(\overline{y}) \sigma_{2}(|q_{\overline{y}}+\xi|) \geq \sigma_{1}(|q_{\overline{y}}+\xi|)\geq \sigma_{1}(2L_{1}).                 \\
     \end{aligned}
     \right.
\end{equation}
Then from (\ref{the:eq112}), we have
\begin{equation}\label{the:eq113}
\text{D} \geq -\frac{C}{\sigma_{1}(2L_{1})}||f||_{L^{\infty}(B_{1})}.
\end{equation}

Consequently, combining (\ref{the:eq17}) and (\ref{the:eq113}), we get
\begin{align*}
  2 \beta \bigg( -r_{0}^{2-2s} \min \bigg\{1,p-1\bigg\} L_{1} +& C L_{1} o(1) + \frac{CL_{2}}{\alpha_{p}}r_{0}^{2-2s} \bigg)
  +  2\beta C \epsilon_{0}^{-n-2s} \frac{1-s}{\alpha_{p}} ||u||_{L^{1}_{s}(\mathbb{R}^{n})}  \\
  & \ \ \ + \frac{C}{\sigma_{1}(2L_{1})}||f||_{L^{\infty}(B_{1})}  \geq 0 ,
\end{align*}
which is a contradiction, provided $ L_{1} $ is large enough such that $ \sigma_{1}(2L_{1}) \geq \sigma_{1}(1) \geq 1 $.

Thereby, we verify \eqref{Sec6:extreq1}, which implies that $ u $ is locally Lipschitz continuous with the estimate
\begin{equation*}
 |u(x)-u(y)| \leq C |x-y|.
\end{equation*}

Case 2: $ |\xi| \leq A_{0}  $. The procedure is very similar to case 1, we only change the formula (\ref{Sec6:eq123}) to $ h(t) = t^{\beta} $ for some $ \beta \in (0,1) $ and repeat the process to derive $ u \in C^{0,\beta}_{loc}(B_{1}) $. In this way, the proof of Lemma \ref{Sec6:lem1} has been completed.
\end{proof}
\end{Lemma}

The compactness stemming from the former result unlocks an approximation lemma similar to Proposition \ref{prop:3.2}. The ingredient of proof is that, for $ s $ sufficiently close to $ 1 $, $ \Delta_{p,N}^{s} \rightarrow \Delta_{p}^{N} $ and the following canceling property holds
\begin{equation*}
   - \bigg( \sigma_{1}(|Du|)  + a(x) \sigma_{2}(|Du|) \bigg) \Delta_{p}^{N} u  = 0 \ \text{in} \ B_{1}  \ \ \Rightarrow \ \ \Delta_{p}^{N} u  = 0 \ \text{in} \ B_{1},
\end{equation*}
see, for instance, \cite[ Lemma 2.6]{AE18}.

The remaining part is analogous to Section \ref{subsec:3.3}. We omit these details for the sake of simplicity. Therefore, we complete the proof of Theorem \ref{thm4}.

\section*{Acknowledgment}

The authors would like to express their sincere gratitude to the anonymous referee for his/her useful and constructive comments that greatly improved the manuscript. This work was supported by the National Natural Science Foundation of China (No. 12271093) and the Jiangsu Provincial Scientific Research Center of Applied Mathematics (Grant No. BK20233002) and the Start-up Research Fund of Southeast University (No. 4007012503).

\end{sloppypar}
\end{document}